\documentclass[12pt]{amsart}
\pdfoutput=1
\usepackage[margin=1in]{geometry}
\usepackage[utf8]{inputenc}
\usepackage{graphicx}
\usepackage{amsmath,amssymb,hyperref}
\usepackage{amsthm}
\usepackage{color}
\usepackage{tikz}
\usepackage[utf8]{inputenc}
\allowdisplaybreaks

\makeatletter
\def\l@section{\@tocline{1}{12pt plus2pt}{0pt}{}{\bfseries}}
\def\l@subsection{\@tocline{2}{0pt}{2pc}{2pc}{}}
\makeatother
\theoremstyle{plain}
\newtheorem{thm}{Theorem}[section]

\newtheorem{cor}[thm]{Corollary}
\newtheorem{lem}[thm]{Lemma}
\newtheorem{prop}[thm]{Proposition}

\theoremstyle{definition}
\newtheorem{defn}[thm]{Definition}

\theoremstyle{remark}
\newtheorem{rem}[thm]{Remark}
\newtheorem{obs}[thm]{Observation}
\newtheorem{exa}[thm]{Example}
\theoremstyle{plain}

\numberwithin{equation}{section}
\newtheorem{ques}[thm]{Question}

\theoremstyle{plain} 
\newcommand{\thistheoremname}{}
\newtheorem{genericthm}[thm]{\thistheoremname}

  \newtheorem*{genericthm*}{\thistheoremname}
\newenvironment{namedthm*}[1]
  {\renewcommand{\thistheoremname}{#1}%
   \begin{genericthm*}}
  {\end{genericthm*}}

\newcommand{\eps}{\varepsilon}

\newcommand{\B}{{\mathbb B}}
\newcommand{\D}{{\mathbb D}}

\newcommand{\R}{{\mathbb R}}

\newcommand{\C}{{\mathbb C}}

\newcommand{\N}{{\mathbb N}}

\newcommand{\Z}{{\mathbb Z}}

\newcommand{\calD}{{\mathcal D}}
\newcommand{\calM}{{\mathcal M}}
\newcommand{\calG}{{\mathcal G}}

\newcommand{\dist}{\hbox{ \rm dist}}

\newcommand{\calS}{{\mathcal S}}

\makeatletter
\newcommand{\vast}{\bBigg@{4}}
\newcommand{\Vast}{\bBigg@{5}}
\makeatother

\def\udot#1{\ifmmode\oalign{$#1$\crcr\hidewidth.\hidewidth
    }\else\oalign{#1\crcr\hidewidth.\hidewidth}\fi}

\def\R{\mathbb{R}}
\def\Z{\mathbb{Z}}
\def\T{\mathbb{T}}
\def\C{\mathbb{C}}

\def\beq{\begin{equation}}
\def\eeq{\end{equation}}
\makeatletter
\newcommand{\doublewidetilde}[1]{{%
  \mathpalette\double@widetilde{#1}%
}}
\newcommand{\double@widetilde}[2]{%
  \sbox\z@{$\m@th#1\widetilde{#2}$}%
  \ht\z@=.9\ht\z@
  \widetilde{\box\z@}%
}
\makeatother

\def\one{\mbox{1\hspace{-4.25pt}\fontsize{12}{14.4}\selectfont\textrm{1}}}

\usepackage{tikz}

\makeatletter
\def\@makefnmark{%
  \leavevmode
  \raise.9ex\hbox{\fontsize\sf@size\z@\normalfont\tiny\@thefnmark}}
\makeatother

\begin{document}
	
\title[Hypersingular Operators]{From Complex--Analytic Models to Dyadic Methods: A Real--Variable Approach to Hypersingular Operators}

\author{Bingyang Hu}
\address{(Bingyang Hu) Department of Mathematics and Statistics\\
        Auburn University\\
        Auburn, Alabama, U.S.A, 36849}
\email{bzh0108@auburn.edu}

\author{Xiaojing Zhou}
\address{(Xiaojing Zhou) Department of Mathematics and Statistics\\
         Auburn University\\
         Auburn, Alabama, U.S.A, 36849}
\email{xiz0003@auburn.edu}



\begin{abstract}
Motivated by the work of Cheng--Fang--Wang--Yu on the hypersingular Bergman projection, we develop a real-variable framework for hypersingular operators in regimes where strong-type bounds fail on the critical line. Our main new ingredient is the Forelli--Rudin method: a dyadic mechanism, inspired by complex--analytic Forelli--Rudin type arguments, that yields sharp critical-line and endpoint estimates.

On the unit disc, for $1<t<3/2$, we give a complete $(p,q)$-mapping characterization for the dyadic hypersingular maximal operator $\mathcal M_t^{\mathcal D}$, including sharp bounds on the critical line $1/q-1/p=2t-2$ and a weighted endpoint criterion in the radial setting. We also prove a novel two-weight estimate for $\mathcal M_t^{\mathcal D}$ in the range $p>q$, valid for all $t>0$. For the hypersingular Bergman projection
\[
K_{2t}f(z)=\int_{\mathbb D}\frac{f(w)}{(1-z\overline w)^{2t}}\,dA(w),
\]
we establish sharp critical-line bounds, with emphasis on the endpoint weak-type estimate at $(p,q)=\bigl(\tfrac{1}{3-2t},1\bigr)$. In particular, this result resolves an open question on the critical-line behavior of the Bergman projection in the hypersingular regime. Finally, we introduce a class of hypersingular cousins of sparse operators in $\mathbb R^n$ associated with \emph{graded} sparse families, quantified by the sparseness $\eta$ and a new structural parameter (the \emph{degree}) $K_{\mathcal S}$. We characterize the corresponding sharp strong- and weak-type regimes in terms of $(n,t,\eta,K_{\mathcal S})$.

This real-variable perspective addresses an inquiry of Cheng--Fang--Wang--Yu on developing effective real-analytic tools in the hypersingular regime for both $\calM_t^{\calD}$ and $K_{2t}$, and it also provides a new route to critical-line analysis for Forelli--Rudin type and related hypersingular operators in both
real and complex settings.

\end{abstract}

\date{\today}
\maketitle

\tableofcontents

\section{Introduction} \label{20251219sec01}

The present paper is motivated by the recent work of Cheng, Fang, Wang, and Yu \cite{CFWY2017}, who studied the following Bergman-type operator on the unit disc $\D$: for $t>0$,
\begin{equation} \label{20251228defn45}
K_{2t}f(z):=\int_{\D}\frac{f(w)}{(1-z\overline{w})^{2t}}\,dA(w),
\end{equation} 
where $dA$ denotes the normalized area measure on $\D$. The study of the operator $K_{2t}$ can, in broad terms, be divided into three regimes:
\begin{enumerate}
    \item $t=1$, in which case $K_{2t}=K_{2}$ coincides with the Bergman projection on $\D$;
    \item $0<t<1$, in which case $K_{2t}$ is the fractional Bergman projection on $\D$;
    \item $t>1$, in which case $K_{2t}$ becomes the hypersingular Bergman projection on $\D$.
\end{enumerate}
In the first two regimes, the situation is fairly well understood. One \emph{key} reason is that the Bergman projection can be viewed as a generalized Calder\'on--Zygmund operator (see, e.g., \cite{McNeal94}). Consequently, one may bring to powerful tools from Calder\'on--Zygmund theory and dyadic harmonic analysis (such as sparse domination) in the study of the Bergman projection, as well as its fractional counterparts. This has proved to be a fruitful line of research, encompassing (among many other directions) 
\begin{enumerate}
\item [$\bullet$] weight theory \cite{BB1978, PR2013, RTW17, Sehba2018}, 

\item [$\bullet$] Bergman theory in several complex variables \cite{Barrett1992, Fefferman1974, LS2012, MS1994, PS1977, WW2021}, and

\item [$\bullet$] commutator and $BMO$ theory \cite{DLLW2025, HHLPW2024, LiLu1994, Zhu1992}. 
\end{enumerate} 
\noindent We emphasize that the literature in each of these directions is extensive, and the above list is far from exhaustive and is included only for the reader's convenience.

\medskip 

For the third regime $t>1$, to the best of our knowledge, the existing results rely mainly on complex and functional analytic methods, more precisely within the framework of Forelli--Rudin type operators, dating back to the early work of Forelli and Rudin \cite{FR1974} in 1974. We refer the reader to \cite{CFWY2017, ZZ2022} for more recent developments.

A common feature of these results is that they provide strong $L^{p}$--$L^{q}$ bounds \emph{only} away from the critical line, while estimates on the critical line itself appear to be unavailable in the literature. One main reason is that strong-type bounds on the critical line generally \emph{fail} (see, Figure~\ref{Fig2}). This obstruction, in turn, make it difficult to use the techniques that are effective in the case $t\le 1$, since $K_{2t}$ becomes more singular in the hypersingular regime.

\begin{rem}
It is important to distinguish the \emph{hypersingular} operators studied in this paper from the \emph{strongly singular Calder\'on--Zygmund operators}\footnote{The authors thank \'Arp\'ad B\'enyi for pointing out this distinction.} introduced by Alvarez and Milman \cite{AM1986}. The latter are formulated under the a priori assumption that the operator extends boundedly on $L^2(\R^n)$, together with suitable off-diagonal regularity of the kernel.

In contrast, the hypersingular operators considered here do \emph{not} admit an $L^2$ theory; moreover, strong-type bounds typically fail not only at the diagonal endpoint but even along the critical line (see, e.g., Figures \ref{Fig1}, \ref{Fig2}, and \ref{Fig3} below). Consequently, hypersingular operators exhibit a genuinely different type of singular behavior from the strongly singular Calder\'on--Zygmund class.
\end{rem}

\begin{rem}
In \cite[Section~1]{CFWY2017}, the authors remarked that in the hypersingular regime $t>1$, it remains an open direction to develop a real-variable approach that yields explicit boundedness results for concrete operators. This is one main motivation for the present work.

\end{rem}

The first \emph{goal} of the current paper is to develop harmonic-analytic methods for the study of the hypersingular Bergman projections $K_{2t}$ for $t>1$, thanks to the recent development in dyadic harmonic analysis. In particular, we aim to establish estimates on the critical line, which, to the best of our knowledge, are even new in the unit disc setting. 

\medskip 

Let us now turn to some details. Instead of studying $K_{2t}$ for $t>1$ directly, it is natural from the viewpoint of harmonic analysis to first consider its maximal analogue, which we call the \emph{hypersingular maximal operator}.

\begin{defn}
For $t>1$, the \emph{hypersingular maximal operator} $\mathcal{M}_t$ is defined by
\[
\mathcal{M}_t f(z)
:= \sup_{\substack{I\subseteq \T \\ z\in Q_I}}
\frac{1}{|Q_I|^{t}} \int_{Q_I} |f(w)|\, dA(w),
\]
where $\T$ denotes the unit circle and $Q_I$ is the Carleson box associated with an arc $I\subseteq\T$. It is also convenient to introduce its dyadic analogue. Let $\mathcal{D}$ be a dyadic system on $\T$. For $t>1$, the \emph{dyadic hypersingular maximal operator} $\mathcal{M}_t^{\mathcal{D}}$ is defined by
\[
\mathcal{M}_t^{\mathcal{D}} f(z)
:= \sup_{\substack{I\in\mathcal{D} \\ z\in Q_I}}
\frac{1}{|Q_I|^{t}} \int_{Q_I} |f(w)|\, dA(w).
\]
\end{defn}

\begin{rem} \label{202525rem01}
In what follows, we will restrict our attention to the dyadic hypersingular maximal operator $\mathcal{M}_t^{\mathcal{D}}$. It is a standard fact (via the $1/3$-trick) that $\mathcal{M}_t$ is pointwise comparable to the sum of two dyadic counterparts: there exists two dyadic systems $\mathcal{D}$ and $\widetilde{\mathcal{D}}$ on $\T$ such that\footnote{Here we refer to $\calD$ and $\widetilde{\calD}$ as a pair of adjacent dyadic systems on $\T$.
}
\[
\mathcal{M}_t f(z)\simeq  \mathcal{M}_t^{\mathcal{D}} f(z)+\mathcal{M}_t^{\widetilde{\mathcal{D}}} f(z), \quad z \in \D.
\]
Here $\mathcal{D}$ may be taken to be the standard dyadic system on $\T$, and $\widetilde{\mathcal{D}}$ is the $1/3$-shifted dyadic system. This reduction dates back to the work of Garnett, Jones, and Mei on $BMO$ and its dyadic counterpart (see \cite{GJ1982, TM2003}).
\end{rem}

We first observe that it only makes sense to consider the case when $1<t<3/2$ in the hypersingular regime. Indeed, it is easy to check that 
$$
\left(\calM_t^{\calD} 1 \right)(z)=\sup_{z \in Q_I,\; I \in \calD } \frac{|Q_I|}{|Q_I|^t} \simeq \frac{1}{(1-|z|^2)^{2(t-1)}}, \qquad z \in \D. 
$$
The minimal requirement here is to make $\calM_t^{\calD} 1 \in L^1(\D)$, and hence 
\begin{equation} \label{20251218eq01a}
\int_{\D} \left| \left( \calM_t^{\calD} 1 \right)(z) \right| dA(z) \simeq \int_{\D} \frac{1}{(1-|z|^2)^{2(t-1)}} dA(z) \simeq \int_0^1 \frac{dr}{(1-r)^{2(t-1)}},
\end{equation} 
which is finite if and only if $t<3/2$. 

\begin{rem} \label{20251227rem01}
The behavior of $\calM_t^{\calD}$ is more subtle than that of the classical maximal operator over Carleson tents and its fractional analogue. In particular, it depends not only on the \emph{underlying dyadic structure}, but also on the \emph{geometry of the ambient domain}. More precisely:
    \begin{enumerate}
        \item In the real-variable setting (namely, for the classical Hardy--Littlewood maximal operator), it is clear that it is not meaningful to consider $t>1$, since in that range the operator is \emph{not} even well-defined on nonzero constant functions.

        \item Likewise, if one replaces $\D$ by the upper half plane $\mathcal H:=\{z \in \mathbb C: \textnormal{Im} z>0 \}$, then a careful inspection of the argument in \eqref{20251218eq01a} yields that the operator is again \emph{not} well-defined on nonzero constant functions whenever $t>1$ (see also \cite[Theorem~5]{CFWY2017} for a related statement concerning the hypersingular Bergman projection $K^{\mathcal H}_{2t}$ on the upper half plane).

    \end{enumerate}
\end{rem}

We have the following \emph{full} characterization for $\calM_t^{\calD}$, including the critical line estimates. Let $1<t<3/2$ and $\calD$ be any dyadic system on $\T$. Then $\calM_t^{\calD}$ (as well as $\calM_t$) is 
\begin{enumerate}
    \item {\bf (Strong-type bound)} bounded from $L^p(\D)$ to $L^q(\D)$ for 
    $$
    \frac{1}{q}-\frac{1}{p}>2t-2, \quad \textrm{with} \ 1 \le p, q \le +\infty
    $$
    (see, Proposition \ref{20251228prop01});

    \vspace{0.1cm}
    
    \item  {\bf (Weak-type bound)} bounded from $L^p(\D)$ to $L^{q, \infty} (\D)$ for 
    $$
    \frac{1}{q}-\frac{1}{p}=2t-2, \quad \textrm{with} \ 1 \le p, q \le +\infty
    $$
    (see, Lemma \ref{20251219lem02}, Theorem \ref{20251222thm01}, and Corollary \ref{20251226lem10}). 
\end{enumerate}
Moreover, the above estimates are \emph{sharp} (see, Proposition \ref{20260131prop01}). The results above are summarized in Figure \ref{Fig1} below.

\begin{center}
\begin{figure}[!htbp]
\begin{tikzpicture}[scale=4]

\def\t{1.25} 

\coordinate (A) at (0,{2*\t-2});   
\coordinate (B) at ({3-2*\t},1);   

\draw[->, black, thick] (-0.12,0) -- (1.20,0);
\draw[->, black, thick] (0,-0.12) -- (0,1.20);

\node[below]       at (1,0) { \footnotesize$1$};
\node[below right] at (1.1,0) {\footnotesize $\frac{1}{p}$};

\node[left]        at (0,1) {\footnotesize $1$};
\node[above left]  at (0,1.1) {\footnotesize $\frac{1}{q}$};

\draw[black, very thick] (0,0) rectangle (1,1);

\fill (0,0) circle (0.015);
\fill (1,0) circle (0.015);
\fill (1,1) circle (0.015);

\fill[red!30] (0,1) -- (B) -- (A) -- cycle;

\draw[red!45!red, line width=2pt] (0,1) -- (A);
\draw[red!45!red, line width=2pt] (0,1) -- (B);

\draw[blue!45!blue, line width=2.5pt] (A) -- (B);

\fill[blue!45!blue] (A) circle (0.02);
\fill[blue!45!blue] (B) circle (0.02);

\node[left]  at (A) {\tiny $(0,\,2t-2)$};
\node[above] at (B) {\tiny $(3-2t,\,1)$};
\end{tikzpicture}
\caption{\footnotesize{Boundedness of $\calM_t^{\calD}$ for $1<t<3/2$: the red line and the shaded region indicate strong $(p,q)$ bounds, while the blue line indicates weak $(p,q)$ bounds.}}
\label{Fig1}
\end{figure}
\end{center} 

\vspace{-0.3in}

Moreover, at the endpoint $\left(p, q\right)=\left(\frac{1}{3-2t},1\right)$, we obtain necessary and sufficient conditions for the endpoint bounds of $\calM_t^{\calD}$ in the radial weighted setting $\omega(z)=\omega(|z|)$, $z\in\D$. More precisely, we prove the following characterizations:
\begin{enumerate}
\item \textbf{Weak-type bounds characterization} (see, Theorem \ref{20251222thm01}): 
\[
\calM_t^{\calD}: L^p(\D,\omega)\to L^{q,\infty}(\D) \ \textrm{is bounded}
\ \Longleftrightarrow\
\sup_{k \ge 0} 2^k \int_{1-\frac{1}{2^k}}^{1-\frac{1}{2^{k+1}}}
\omega(r)^{-\frac{3-2t}{2t-2}}\,dr<+\infty
\]
\item \textbf{Strong-type bounds characterization} (see, Theorem \ref{20251222thm02}): assume in addition that $\omega \in {\bf B}_{\frac{1}{3-2t}}$. Then
\[
\calM_t^{\calD}: L^p(\D,\omega)\to L^{q}(\D) \ \textrm{is bounded}
\ \Longleftrightarrow\
\sum_{k \ge 0} 2^k \int_{1-\frac{1}{2^k}}^{1-\frac{1}{2^{k+1}}}
\omega(r)^{-\frac{3-2t}{2t-2}}\,dr<+\infty.
\]
Here, ${\bf B}_{\frac{1}{3-2t}}$ denotes the B\'ekoll\'e--Bonami weight class (see, \eqref{20251228defn21}). 
\end{enumerate}

To this end, we consider weighted estimates for $\calM_t^{\calD}$, with the main focus on the regime $\{1\le p,q \le \infty: p>q\}$. In this range, we obtain a somewhat more general two-weight estimate for $\mathcal{M}_t^{\mathcal{D}}$ for any $t>0$. More precisely, we show that for $\mu, \omega$ being two weights on $\D$ satisfying certain $\bf B_\infty$ condition (see, Definition \ref{20260119defn01}), $\calM_t^{\calD}: L^p(\omega, \D) \to L^q(\mu, \D)$ is bounded if and only if $\phi \in L^{\frac{p}{p-q}}(\D)$, where
\begin{equation} \label{20260119eq01}
\phi(z):=\sum_{I \in \calD} \beta_I \one_{Q_I^{\textrm{up}}}(z), \quad \textrm{with} \quad \beta_I:= \frac{1}{|Q_I|^{(t-1)q}} \cdot \frac{\mu(Q_I)}{|Q_I|} \cdot \left(\frac{\sigma(Q_I)}{|Q_I|} \right)^{\frac{q}{p'}},
\end{equation}
where $\sigma:=\omega^{-1/(p-1)}$ is the dual weight of $\omega$
(see, Theorem \ref{20260114thm01a}). 

\begin{rem}
\begin{enumerate}
    \item The regime $\{1\le p,q\le \infty: p\le q\}$ in the above result can be handled by standard methods in weighted theory (see, e.g., \cite{CB2016} for the case $t=1$).

    \item The condition \eqref{20260119eq01} may be viewed as a hypersingular counterpart of the B\'ekoll\'e--Bonami condition. Indeed, in the limiting case $p=q$ and $t=1$, \eqref{20260119eq01} reduces exactly to the B\'ekoll\'e--Bonami ${\bf B}_p$ condition.

    \item Condition \eqref{20260119eq01} can also be interpreted from the perspective of Bergman--Carleson embeddings in complex function theory, initiated in a series of influential works of Luecking \cite{Luecking1983,Luecking1991,Luecking1993}. 
    In that setting, the boundedness of the embedding operator $
    id: A^p(\D)\to L^q_\mu(\D)$ (here, $A^p(\D)$ denotes the standard Bergman space on $\D$) is characterized by an $L^{\frac{p}{p-q}}(\D)$-integrability condition on an appropriate testing function. 
    From this viewpoint, the above result also suggests a way to understand the B\'ekoll\'e--Bonami condition through Carleson embedding.
\end{enumerate}
\end{rem}

Our proofs of the above results build on recent developments in dyadic harmonic analysis, together with underlying geometry properties of dyadic Carleson boxes.

\begin{rem}
As pointed out in \cite{CFWY2017} (see the discussion under ``Maximal operators'' there), the maximal operator $\calM_t$ (or $\calM_t^{\calD}$) associated with Carleson boxes on $\D$ arises naturally as a hypersingular analogue of the classical maximal operator, and one expects a corresponding boundedness theory in the range $1<t<3/2$. The results above therefore provide such a theory from a real-variable perspective.

\end{rem}

\medskip 

Next, we consider the behavior of $K_{2t}$. As observed in \cite{CFWY2017}, it suffices to restrict to the range $1<t<3/2$. In this regime, we obtain a \emph{full} characterization of the behavior of $K_{2t}$, including the critical-line behavior. In particular, $K_{2t}$ is 
\begin{enumerate}
    \item {\bf (Strong-type bound)} bounded from $L^p(\D)$ to $L^q(\D)$ for 
    $$
    \frac{1}{q}-\frac{1}{p}>2t-2, \quad \textrm{with} \ 1 \le p, q \le +\infty
    $$
    (see, \cite[Theorem 3]{CFWY2017}, and also Theorem~\ref{20251227thm01}, (1) for a generalization and strengthening of this result in terms of hypersingular sparse operators);

    \vspace{0.1cm}

    \item {\bf (Weak-type bound)} bounded from $L^p(\D)$ to $L^{q, \infty} (\D)$ for 
    $$
    \frac{1}{q}-\frac{1}{p}=2t-2, \quad \textrm{with} \ 1 \le p, q \le +\infty. 
    $$
    (see, Lemma \ref{20251225lem01}, Proposition \ref{20251226thm01},  and Corollary \ref{20251228cor01}).  
\end{enumerate}
Moreover, the above estimates for $K_{2t}$ are \emph{sharp} (see, Proposition \ref{20260131prop02}). We summarize the above results for $K_{2t}$ in Figure \ref{Fig2} below.

\vspace{-0.3cm}

\begin{center}
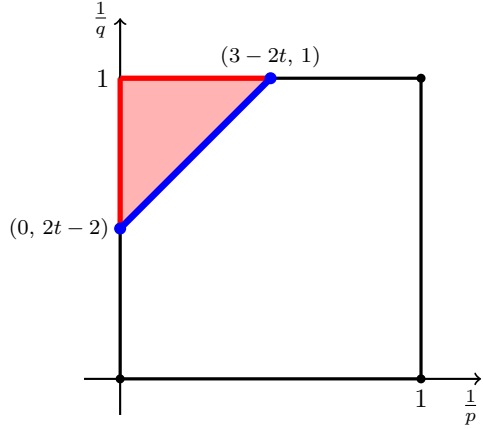
\begin{figure}[!htbp]
\begin{tikzpicture}[scale=4]

\def\t{1.25} 

\coordinate (A) at (0,{2*\t-2});   
\coordinate (B) at ({3-2*\t},1);   

\draw[->, black, thick] (-0.12,0) -- (1.20,0);
\draw[->, black, thick] (0,-0.12) -- (0,1.20);

\node[below]       at (1,0) { \footnotesize$1$};
\node[below right] at (1.1,0) {\footnotesize $\frac{1}{p}$};

\node[left]        at (0,1) {\footnotesize $1$};
\node[above left]  at (0,1.1) {\footnotesize $\frac{1}{q}$};

\draw[black, very thick] (0,0) rectangle (1,1);

\fill (0,0) circle (0.015);
\fill (1,0) circle (0.015);
\fill (1,1) circle (0.015);

\fill[red!30] (0,1) -- (B) -- (A) -- cycle;

\draw[red!45!red, line width=2pt] (0,1) -- (A);
\draw[red!45!red, line width=2pt] (0,1) -- (B);

\draw[blue!45!blue, line width=2.5pt] (A) -- (B);

\fill[blue!45!blue] (A) circle (0.02);
\fill[blue!45!blue] (B) circle (0.02);

\node[left]  at (A) {\tiny $(0,\,2t-2)$};
\node[above] at (B) {\tiny $(3-2t,\,1)$};
\end{tikzpicture}
\caption{\footnotesize{Boundedness of $K_{2t}$ for $1<t<3/2$: the red line and the shaded region indicate strong $(p,q)$ bounds, and the blue line indicates weak $(p,q)$ bounds.
}}
\label{Fig2}
\end{figure}
\end{center} 

\vspace{-0.2in}

One \emph{main} difficulty in analyzing $K_{2t}$ occurs at the endpoint $p=\frac{1}{3-2t}$ and $q=1$ (see, Proposition \ref{20251226thm01}). As noted in \cite[Theorem~3]{CFWY2017}, strong-type bounds fail along the critical line $1/q-1/p=2t-2$, so $K_{2t}$ behaves more singularly in this regime. This is in sharp contrast with the Bergman projection $K_{2}$, whose $L^{2}$ boundedness is is simply guaranteed by its definition.

To address this issue, we introduce a new approach, which we call the \emph{Forelli--Rudin method}.
For the endpoint behavior of $K_{2t}$ at $(p,q)=\bigl(\tfrac{1}{3-2t},1\bigr)$, we will use its complex--analytic version, which serves as a motivating case.

\medskip 

\noindent {\bf Key idea 1: A complex--analytic version of the Forelli--Rudin method (motivating case).}
To treat $K_{2t}$ at the endpoint $(p,q)=\bigl(\tfrac{1}{3-2t},1\bigr)$, the main idea is to use the factorization
\begin{equation} \label{20260127eq50a}
K_{2t} f(z)=W(z)\,B_{2t}f(z), \qquad z\in\D,
\end{equation} 
where $W(z)=(1-|z|^2)^{2-2t}$ and $B_{2t}f(z):=(1-|z|^2)^{2t-2}K_{2t}f(z)$.
This separates the singularity into a simple weight $W$ and a less singular Forelli--Rudin type operator $B_{2t}$, thereby reducing the endpoint weak-type bound to a strong-type estimate for $B_{2t}$ together with a borderline Lorentz control of $W$.

\begin{rem} \label{20260106rem01}
A careful examination of the proof of the above results for $K_{2t}$ shows that the same bounds also hold for the associated positive operator 
$$
K_{2t}^{+}f(z):=\int_{\D}\frac{f(w)}{|1-z\overline{w}|^{2t}}\,dA(w),
$$
which can be regarded as a hypersingular analogue of the Berezin transform.  
\end{rem}

Motivated by our approaches to $\calM_t^{\calD}$ and $K_{2t}$, it is natural to ask whether these methods can be extended to study hypersingular analogues of sparse operators in harmonic analysis. This leads to the second main \emph{goal} of the present paper. More precisely, let $t>1$ and let $\calS$ be a sparse family in $\R^n$ such that\footnote{This global containment assumption is natural in the hypersingular setting (see Remark~\ref{20251229rem02} for further discussion).} there exists a dyadic cube $Q_0$ with $Q\subseteq Q_0$ for all $Q\in\calS$. Consider the following \emph{hypersingular averaging operator}
\[
\mathbb A_{\calS}^t f(x)
:=\sum_{Q\in\calS}\frac{\one_{Q}(x)}{|Q|^{t}}\int_Q |f(y)|\,dy.
\]
We are interested in determining the admissible $(p, q)$-range for which $\mathbb A_{\calS}^t$ is of strong type, or weak type. Note that the study of $\mathbb A_{\calS}^t$ is also of independent interest from the viewpoint of dyadic harmonic analysis, since it can be viewed as a hypersingular counterpart of the classical sparse operator.

It turns out that the boundedness behavior of $\mathbb A_{\calS}^t$ is determined by the following \emph{four} parameters:
\begin{enumerate}
    \item $n \ge 1$, the \emph{real dimension} of the ambient space; 
    \item $t>1$, the \emph{hypersingular index} of the averaging operator  $\mathbb A_{\calS}^t$;
    \item $\eta \in (0, 1)$, the \emph{sparseness} of $\calS$,  measuring how much of each cube can be chosen disjointly (see, Definition \ref{20251227defn01a}).
    \item $K_{\calS} \ge 1$, the \emph{degree} of $\calS$, rouphly speaking, which measures the maximal dyadic scale drop between consecutive layers (see, Definition \ref{20251227defn01}).
\end{enumerate}
We note that, in the hypersingular regime, it is \emph{pivotal} to assume that the degree $K_{\calS}$ is \emph{finite}. We have the following motivating example. 

\begin{exa} \label{20251227obs01}
Let $t>1$ and 
$$
\calS_m:=\left\{ I_k:=\left[ \frac{k}{2^m}, \frac{k+1}{2^m} \right), 0 \le k \le 2^m-1 \right\} \cup \left\{ [0, 2) \right\},
$$
for each $m \ge 1$. It is clear that $\calS_m$ is an $\eta$-sparse family for any $\eta \in (1/2, 1)$ and any $m \ge 1$; however
$$
\mathbb A_{\calS_m}^t 1(x) \ge \sum_{k=0}^{2^m-1} \frac{\one_{I_k}(x)}{|I_k|^t} \int_{I_k} 1 dx=2^{m(t-1)},  \quad x \in [0, 1), 
$$
which implies $\left\|\mathbb A_{\calS_m}^t \right\|_{L^\infty([0. 2)) \to L^1([0, 2))} \gtrsim 2^{m(t-1)}$. This example shows that, in order to capture the behavior of $\mathbb A_{\calS}^t$, the sparseness of $\calS$ alone does not suffice, and one must have certain control for the dyadic scales between consecutive layers in $\calS$. This simple example also highlights a fundamental difference between the sparse operator and its hypersingular counterpart: in the former case, the degree plays essentially no role. We refer the reader to Section~\ref{20251229subsec01} for further discussion of this topic.
\end{exa}

Here are the main results in the second part of the paper. Let $\calS$ be a sparse family in $\R^n$ with sparseness $\eta\in(0,1)$ and degree $K_{\calS}\in[1,\infty)$; we refer to such a family as a \emph{graded sparse family}. Then for any $1<t<1-\frac{\log_2(1-\eta)}{nK_{\calS}}$, we have $\mathbb A_{\calS}^t$ is 
\begin{enumerate}
    \item {\bf (Strong-type bound)} bounded from $L^p(\D)$ to $L^q(\D)$ for 
    $$
    \frac{1}{q}-\frac{1}{p}>\frac{nK_{\calS}(t-1)}{-\log_2(1-\eta)}, \quad \textrm{with} \ 1 \le p, q \le +\infty;
    $$

    \vspace{0.1cm}

    \item {\bf (Weak-type bound)} bounded from $L^p(\D)$ to $L^{q, \infty} (\D)$ for 
    $$
    \frac{1}{q}-\frac{1}{p}=\frac{nK_{\calS}(t-1)}{-\log_2(1-\eta)}, \quad \textrm{with} \ 1 \le p, q \le +\infty. 
    $$ 
\end{enumerate}
We will prove these estimates in Theorem~\ref{20251227thm01}. For the reader’s convenience, we summarize the resulting bounds below (see Figure~\ref{Fig3}).

\begin{center}
\begin{figure}[!htbp]
\begin{tikzpicture}[scale=4]

\def\t{1.25} 

\coordinate (A) at (0,{2*\t-2});   
\coordinate (B) at ({3-2*\t},1);   

\draw[->, black, thick] (-0.12,0) -- (1.20,0);
\draw[->, black, thick] (0,-0.12) -- (0,1.20);

\node[below]       at (1,0) { \footnotesize$1$};
\node[below right] at (1.1,0) {\footnotesize $\frac{1}{p}$};

\node[left]        at (0,1) {\footnotesize $1$};
\node[above left]  at (0,1.1) {\footnotesize $\frac{1}{q}$};

\draw[black, very thick] (0,0) rectangle (1,1);

\fill (0,0) circle (0.015);
\fill (1,0) circle (0.015);
\fill (1,1) circle (0.015);

\fill[red!30] (0,1) -- (B) -- (A) -- cycle;

\draw[red!45!red, line width=2pt] (0,1) -- (A);
\draw[red!45!red, line width=2pt] (0,1) -- (B);

\draw[blue!45!blue, line width=2.5pt] (A) -- (B);

\fill[blue!45!blue] (A) circle (0.02);
\fill[blue!45!blue] (B) circle (0.02);

\node[left]  at (A) {\tiny $\left(0, \frac{nK_{\calS}(t-1)}{-\log_2(1-\eta)} \right)$};
\node[above] at (B) {\tiny $\left(\frac{-\log_2(1-\eta)+nK_{\calS}(1-t)}{-\log_2(1-\eta)}, 1 \right)$};
\end{tikzpicture}
\caption{\footnotesize{Boundedness of $\mathbb A_{\calS}^t$ for $1<t<1-\frac{\log_2(1-\eta)}{nK_{\calS}}$: the red line and the shaded region indicate strong $(p,q)$ bounds, and the blue line indicates weak $(p,q)$ bounds.
}}
\label{Fig3}
\end{figure}
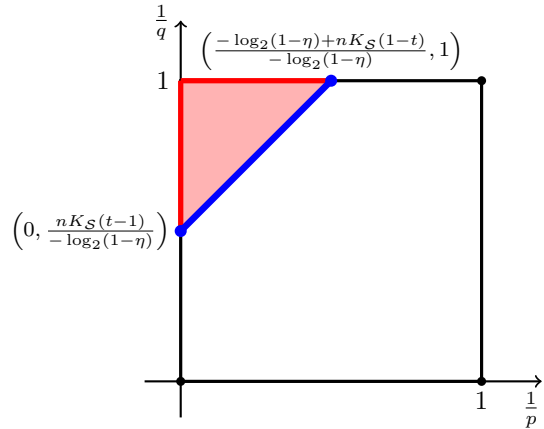
\end{center} 

\vspace{-0.3in}

The crux of the analysis for $\mathbb A_{\calS}^t$ is to establish the weak-type estimate at the endpoint
\begin{equation} \label{20260128eq02}
p=\frac{-\log_2(1-\eta)}{-\log_2(1-\eta)+nK_{\calS}(1-t)}
\qquad\text{and}\qquad
q=1.
\end{equation}
The difficulties at this endpoint are two--fold:
\begin{itemize}
\item As for $K_{2t}$, strong-type bounds fail along the critical line for $\mathbb A_{\calS}^t$.  Consequently, both the Calder\'on--Zygmund approach (for instance, the Calder\'on--Zygmund decomposition, which relies on an $L^{r}$ bound, $1<r<\infty$, to control the bad function) and more recent sparse domination techniques (e.g.\ the weak-type machinery in \cite[Theorem~E]{CCDO2017}) are more delicate to implement in this hypersingular setting.

Indeed, we present two independent proofs, one based on sparse domination machinery and the other using Bourgain's interpolation trick, showing that at the endpoint \eqref{20260128eq02} the operator $\mathbb A_{\calS}^t$ satisfies a \emph{restricted weak-type estimate}. While this is weaker than the weak-type bound proved in Theorem~\ref{20251227thm01}, it already suffices to obtain weak-type estimates on the critical line away from this endpoint. We refer to Section~\ref{20260127subsec01} for details.

\item In this dyadic setting, we also lack an underlying complex-analytic geometry. More precisely, in several complex variables, quantities such as $1-|z|$ are typically interpreted as the boundary distance $\dist(z,\T)$. In contrast, for $\mathbb A_{\calS}^t$ there is no canonical notion of boundary distance, and it is unclear what should play the role of the ``boundary'' of the region.

\end{itemize}

\vspace{0.1cm}

\noindent {\bf Key idea 2: A dyadic version of the Forelli--Rudin method.}
To exploit the idea in \eqref{20260127eq50}, the \emph{key} observation is that the boundary distance $1-|z|$ admits a convenient dyadic re-encoding. Fix $z\in\D$ and let $\calD$ be any dyadic system on $\T$. Define
\[
{\bf N}(z):=\#\{\, I\in\calD:\ z\in Q_I\,\},
\qquad z\in\D,
\]
namely, the number of dyadic Carleson boxes $Q_I$ (with $I\in\calD$) that contain $z$. Then
\[
{\bf N}(z)\simeq \log_2 \frac{1}{1-|z|},
\]
and hence one may view
\[
1-|z| \simeq 2^{-{\bf N}(z)}.
\]
This identity serves as a \emph{bridge} between analysis on holomorphic function spaces and dyadic real-variable harmonic analysis.
Indeed, its left-hand side is a boundary distance from the intrinsic geometry of $\D$, while its right-hand side is a counting
quantity for dyadic Carleson boxes.  It allows us to replace the boundary scale by a dyadic level, and it is the starting point
for the dyadic Forelli--Rudin method.

Using now this idea for $\mathbb A_{\calS}^t$, one can prove the pointwise bound
\begin{equation}\label{20260128eq01}
\mathbb A_{\calS}^t f(x) \lesssim {\bf W}(x)\, M^{\calD}_{\mathrm{HL}}f(x),
\qquad x\in\R^n,
\end{equation}
where ${\bf W}$ is a suitable weight function and $M^{\calD}_{\mathrm{HL}}$ denotes the dyadic Hardy--Littlewood maximal operator on $\R^n$ associated to $\calD$. \eqref{20260128eq01} is precisely the dyadic analogue of \eqref{20260127eq50a} (see Theorem~\ref{20251227thm01}, {\bf Step~III} for more details).
The weak-type bound at the endpoint \eqref{20260128eq02} then follows by establishing the appropriate Lorentz-space estimates for ${\bf W}$.

\medskip 

We make some further remarks. 

\begin{rem}
\begin{enumerate}
\item A noteworthy feature of Theorem~\ref{20251227thm01} is that the geometric parameters of the family---in particular, the sparseness $\eta$ and the degree $K_{\calS}$---enter the \emph{admissible} $(p,q)$-range through the critical relation
\[
\frac1q-\frac1p=\frac{nK_{\calS}(t-1)}{-\log_2(1-\eta)}.
\]
In other words, the geometry of the underlying collection directly affects the \emph{range} of $L^p$--$L^q$ boundedness for $\mathbb A_{\calS}^t$. This is in sharp contrast with the classical sparse operators (corresponding to $t=1$), where sparseness influences only the operator norm, while the set of admissible exponents $(p,q)$ is independent of the particular sparse family.

\item The results above strengthen our earlier characterization of the hypersingular Bergman projection
$K_{2t}$ in the range $1<t<3/2$. Indeed, this follows from the pointwise sparse bounds
$$
\left| K_{2t} f(z) \right| \lesssim \mathbb A_{\{Q_I\}_{I \in \calD}}^t f(z)+\mathbb A_{\{Q_I\}_{I \in \widetilde{\calD}}}^t f(z), \qquad z \in \D, 
$$
where we recall $(\calD, \widetilde{\calD})$ is any pair of adjacent dyadic systems on $\T$ (see, Remark \ref{202525rem01}), and the fact that $\{Q_I\}_{I \in \calD}$ (in real dimension $n=2$) forms a graded sparse family with sparseness $\eta=1/2$ and degree $K_{\calS}=1$.
With these parameters, one has
\[
\frac{nK_{\calS}(t-1)}{-\log_2(1-\eta)}=2t-2,
\quad\text{and}\quad
\frac{-\log_2(1-\eta)}{-\log_2(1-\eta)+nK_{\calS}(1-t)}=\frac{1}{3-2t},
\]
which exactly coincide with those arising in the mapping theory of $K_{2t}$.

\vspace{0.1cm}

\item The above estimates for \(\mathbb A_{\calS}^t\) are \emph{sharp}, as is already reflected in our earlier strong and weak-type characterizations of \(\calM_t^{\calD}\) and \(K_{2t}\). These observations suggest that the $L^p$--$L^q$ mapping properties of $K_{2t}$ are governed primarily by the geometry of (dyadic) Carleson boxes, rather than analyticity.
\vspace{0.1cm}

\item In the results above, the sparseness parameter $\eta$ is understood \emph{with respect to the underlying dyadic grid}.
This is the reason that the quantity $-\log_2(1-\eta)$ (and hence a base-$2$ logarithm) appears naturally in our main theorem.
For a discussion of how this normalization behaves under changing the base of the grid (e.g.\ dyadic versus triadic systems),
see Remark~\ref{20251230rem01}.

\end{enumerate} 
\end{rem}

We conclude the introduction by summarizing the main novelties of the present paper, which can be divided into the following aspects.

\vspace{0.1cm}

\begin{enumerate}
\item[(a)] 
From the viewpoint of complex function theory, we develop a systematic theory of the hypersingular maximal operator $\calM_t^{\calD}$ and also advance the study of Forelli--Rudin type operators in the hypersingular regime by establishing new sharp endpoint and critical-line estimates for $K_{2t}$ with $t>1$. To the best of our knowledge, these estimates are new even in the unit disc setting. A \emph{key} feature of our approach is the introduction of Forelli--Rudin method into this setting, which allows us to overcome the lack of strong-type bounds and provides a robust harmonic-analytic framework for hypersingular operators in both real and complex settings. 

\medskip

\item[(b)] 
From the perspective of harmonic analysis, we introduce a new dyadic model, which we refer to as a \emph{hypersingular sparse operator}. 
Unlike classical sparse operators associated with Calder\'on--Zygmund theory or recent sparse domination frameworks, the $(p,q)$-boundedness range of these hypersingular sparse operators depends intrinsically on the geometry of the underlying sparse family.
In particular, geometric parameters such as sparseness and degree directly influence the admissible $(p, q)$-range, rather than merely affecting operator norms. 
This reveals a genuinely new phenomenon beyond both classical Calder\'on--Zygmund theory and standard sparse domination.

\medskip

\item[(c)] 
More broadly, the Forelli--Rudin methed developed in this paper provides a flexible and systematic tool for treating hypersingular operators.
We expect that this approach will have further applications,  including weighted estimates, commutator theory, hypersingular Bergman theory in several complex variables (on general domains), and related problems, both in complex analysis and in other areas of analysis; see Section~\ref{20251219sec06} for possible directions for future work.
\end{enumerate}

\medskip 

The rest of this paper is organized as follows. In Section \ref{20251219sec02}, we collect basic definitions and notation that will be used throughout the paper. In Section \ref{20251219sec03}, we establish the sharp $L^p$ mapping theory for the dyadic hypersingular maximal operator $\calM_t^{\mathcal D}$, including both the off-critical-line bounds and the critical-line estimates, together with the endpoint characterizations in the radial weighted setting. Moreover, we prove a two-weight estimate for $\calM_t^{\calD}, t>0$. In Section \ref{20251219sec04}, we turn to the hypersingular Bergman projection $K_{2t}$ and prove the sharp critical-line estimates, with special emphasis on the endpoint $(p,q)=\left(\frac{1}{3-2t},1\right)$ via a complex--analytic version of the Forelli--Rudin method. In Section \ref{20251219sec05}, we introduce hypersingular sparse operators associated with graded sparse families in $\mathbb R^n$ and prove the corresponding $L^p$-$L^q$ theory via a dyadic version of the Forelli--Rudin method, which in particular extends and strengthens the results obtained earlier for $\calM_t^{\calD}$ and $K_{2t}$. In addition, we include a discussion of two other different approaches to study $\mathbb A_{\calS}^t$: one via ideas from sparse domination, and the other via Bourgain's interpolation trick. Finally, in Section \ref{20251219sec06}, we propose several open problems and further directions motivated by the methods developed in this paper.

Throughout this paper, for $a ,b \in  \mathbb{R}$, $a\lesssim b$ means there exists a positive number $C$, which is independent of $a$ and $b$, such that $a\leq C\,b$. Moreover, if both $a \lesssim  b$ and $b\lesssim a$ hold, then we say $a \simeq b$. \\
{\bf Acknowledgement.} The authors are grateful to \'Arp\'ad B\'enyi for drawing their attention to strongly singular Calder\'on--Zygmund operators. They also thank Yongsheng Han, Cody Stockdale, Yun--Hao Lee, Kenan Zhang, and Zipeng Wang for helpful discussions during various stages of this work. The first author was supported by the Simons Travel grant MPS-TSM-00007213.

\medskip 

\section{Notations} \label{20251219sec02}

In this section, we collect several basic definitions and notations that will be used frequently throughout the paper.

\vspace{0.1cm}

Let $I\subseteq\T$ be an arc. Define the \emph{Carleson box} associated with $I$ by
\[
Q_I:=\left\{z\in\D:\ \frac{z}{|z|}\in I,\ 1-|I|\le |z|<1\right\},
\]
and also the corresponding \emph{upper-half tent} by
\[
Q_I^{\mathrm{up}}:=\left\{z\in\D:\ \frac{z}{|z|}\in I,\ 1-|I|\le |z|<1-\frac{|I|}{2}\right\}.
\]
A \emph{dyadic system} (or \emph{dyadic grid}) $\calD$ on $\T$ is a collection of arcs
\[
\calD=\bigcup_{k\ge 0}\calD_k,
\]
where each \emph{generation} $\calD_k$ consists of $2^k$ disjoint arcs of equal length $2^{-k}$
whose union is $\T$ (equipped with the normalized arc measure) and such that each arc $I\in\calD_k$ is the union of two arcs
$I^{(1)},I^{(2)}\in\calD_{k+1}$ (called the \emph{dyadic children} of $I$).
Equivalently, after the identification $\T\simeq[0,1)$, one may take
\begin{equation} \label{20251230eq01}
\calD_k=\Bigl\{\Bigl[\frac{m}{2^k},\frac{m+1}{2^k}\Bigr): \ m=0,1,\dots,2^k-1\Bigr\},
\end{equation} 
and view each such interval as an arc on $\T$.

\vspace{0.1cm}

Next, we record a few basic estimates and conventions. Let $1\le p,q\le+\infty$, and let $\mathcal{T}$ denote a sublinear operator on $\D$.
The strong-type bound $\mathcal{T}:L^p(\D)\to L^q(\D)$ is understood in the usual sense.

For $1\le p<\infty$, the \emph{weak Lebesgue space} $L^{p,\infty}(\D)$ consists of all measurable functions $f$ on $\D$ such that
\[
\|f\|_{L^{p,\infty}(\D)}
:=\sup_{\lambda>0}\,\lambda\,\big|\{z\in\D:\ |f(z)|>\lambda\}\big|^{1/p}<\infty.
\]
The Lorentz space $L^{p,r}(\D)$ is defined to be the collection of all measurable functions $f$ such that
$$
\left\|f \right\|_{L^{p, r}(\D)}:=\left( p\int_0^\infty \lambda^r \left| \left\{z \in \D: |f(z)|>\lambda \right\} \right|^{\frac{r}{p}}\frac{d\lambda}{\lambda} \right)^{\frac{1}{r}}, \qquad 0<r<+\infty,  
$$
with the usual modification when $p=\infty$.

We say that $\mathcal{T}$ is of \emph{weak type} $(p,q)$, and write $\mathcal{T}:L^p(\D)\to L^{q,\infty}(\D)$ is bounded, if there exists $C>0$ such that for all $f\in L^p(\D)$ and all $\lambda>0$,
\[
\big|\{z\in\D:\ |\mathcal{T}f(z)|>\lambda\}\big|
\le C\,\lambda^{-q}\|f\|_{L^p(\D)}^{\,q},
\]
with the standard modification when $q=\infty$.

Moreover, we say that $\mathcal{T}$ is of \emph{restricted weak type} $(p,q)$ if $\mathcal{T}:L^{p,1}(\D)\to L^{q,\infty}(\D)$ is bounded, 
or equivalently (a well-known fact), if there exists $C>0$ such that for every measurable set $E\subset\D$ with $|E|<\infty$ and all $\lambda>0$,
\[
\big|\{z\in\D:\ |\mathcal{T}\one_E(z)|>\lambda\}\big|
\le C\,\lambda^{-q}|E|^{q/p},
\]
again with the usual modification when $q=\infty$.

Finally, given a nonnegative locally integrable function $\omega$ on $\D$, referred to as a \emph{weight}, we define the \emph{weighted space} $L^p(\D,\omega)$ to be the collection of all measurable functions on $\D$ satisfying $\|f\|^p_{L^p(\D,\omega)}:=\int_{\D}|f(z)|^p\,\omega(z)\,dz<+\infty$, with the usual modification when $p=\infty$.
A weight $\omega$ is called \emph{radial} if $\omega(z)$ depends only on $|z|$, that is, $\omega(z)=\omega(|z|)$ for all $z\in\D$.

\medskip

\section{$L^p$ theory for the hypersingular maximal operator} \label{20251219sec03}

In this section, we divide our analysis of $\calM_t^{\calD}$ into two distinct regimes:
 
\begin{enumerate}
    \item [(1)] Off-critical line regime, that is, when $ (p, q)  \in \left\{\left(\frac{1}{p}, \frac{1}{q} \right) \in  [0, 1]^2: \frac{1}{q}-\frac{1}{p}>2t-2 \right\}$;
    \item [(2)] Critical line regime, that is, when $(p, q) \in \left\{ \left(\frac{1}{p}, \frac{1}{q} \right) \in [0, 1]^2: \frac{1}{q}-\frac{1}{p}=2t-2 \right\}$.
\end{enumerate}
Here and henceforth, we always assume that $1<t<3/2$. 

\subsection{Off-critical line estimate} We begin with the following observation. 

\begin{obs} \label{20251219obs01}
    For any $0<\varepsilon \le 3-2t$, $\calM_t^{\calD}: L^\infty(\D) \to L^{\frac{1}{2t-2+\varepsilon}} (\D)$ is bounded. 
\end{obs}

\begin{proof}
The proof follows from a direct computation. Indeed, 
\begin{align*}
\int_{\D} \left| \calM_t^{\calD} f(z) \right|^{\frac{1}{2t-2+\varepsilon}}dA(z)
& \lesssim \left\|f \right\|^{\frac{1}{2t-2+\varepsilon}}_{L^\infty} \int_{\D} \left| \calM_t^{\calD} 1 (z) \right|^{\frac{1}{2t-2+\varepsilon}} dA(z) \\
& \simeq \left\|f \right\|^{\frac{1}{2t-2+\varepsilon}}_{L^\infty} \int_{\D} \frac{1}{(1-|z|^2)^{\frac{2(t-1)}{2(t-1)+\varepsilon}}}dA(z)\\ 
&\simeq \left\|f \right\|^{\frac{1}{2t-2+\varepsilon}}_{L^\infty} \int_0^1 \frac{dr}{(1-r)^{\frac{2(t-1)}{2(t-1)+\varepsilon}}} \lesssim \left\|f \right\|^{\frac{1}{2t-2+\varepsilon}}_{L^\infty}. 
\end{align*}
From the above arguments, it is not hard to see that $\calM_t^{\calD}$ is unbounded from $L^\infty(\D)$ to $L^{\frac{1}{2t-2}}(\D)$. 
\end{proof}
Next, we prove estimate near the endpoint $\left(1/p, 1/q \right)=(3-2t, 1)$. 

\begin{lem} \label{20251219lem01}
For any $0<\varepsilon \le 3-2t$, $\calM_t^{\calD}: L^{\frac{1}{3-2t-\varepsilon}}(\D) \to L^1(\D)$ is bounded. 
\end{lem}
 
\begin{proof}
Without loss of generality, we may assume $\varepsilon<3-2t$. The case when $\varepsilon=3-2t$ is obvious from the argument in \eqref{20251218eq01a}. 

Let $\alpha>0$ and denote $E_\alpha:=\left\{z \in \D: \calM_t^{\calD} f(z)>\alpha \right\}$. Observe that one can decompose $E_\alpha$ into a collection of maximal and mutually disjoint Carleson boxes $\{Q_{\alpha, i}\}_{i \ge 1}$ such that 
\begin{equation} \label{20251222eq50}
\frac{1}{|Q_{\alpha, i}|^t} \int_{Q_{\alpha, i}} |f(z)|dA(z)>\alpha, \qquad \textnormal{for each} \; \; i \ge 1.
\end{equation} 
Therefore, 
\begin{align} \label{20251218eq01}
\alpha |E_\alpha|&= \alpha \sum_{i=1}^\infty |Q_{\alpha, i}|  \le \sum_{i=1}^\infty \frac{|Q_{\alpha, i}|}{|Q_{\alpha, i}|^{t-1}} \cdot \frac{1}{|Q_{\alpha, i}|}\int_{Q_{\alpha, i}}|f(z)|dA(z)  \nonumber \\
& \le \sum_{i=1}^\infty \frac{1}{|Q_{\alpha, i}|^{t-1}} \int_{Q_{\alpha, i}^{\textnormal{up}}} \calM^{\calD}f(z) dA(z) \nonumber \\
& \le \sum_{i=1}^\infty \frac{1}{|Q_{\alpha, i}|^{t-1}} \left(\int_{Q^{\textnormal{up}}_{\alpha, i}} \left|\calM^{\calD} f(z) \right|^{\frac{1}{3-2t-\varepsilon}} dA(z) \right)^{3-2t-\varepsilon} \left(\int_{Q_{\alpha, i}} dA(z) \right)^{2t+\varepsilon-2} \nonumber  \\
&= \sum_{i=1}^\infty |Q_{\alpha, i}|^{t-1+\varepsilon}\left(\int_{Q^{\textnormal{up}}_{\alpha, i}} \left|\calM^{\calD} f(z) \right|^{\frac{1}{3-2t-\varepsilon}} dA(z) \right)^{3-2t-\varepsilon} \nonumber  \\
& \le \left(\sum_{i=1}^\infty \int_{Q^{\textnormal{up}}_{\alpha, i}} \left|\calM^{\calD} f(z) \right|^{\frac{1}{3-2t-\varepsilon}} dA(z)  \right)^{3-2t-\varepsilon} \cdot \left( \sum_{i=1}^\infty |Q_{\alpha, i}|^{\frac{t-1+\varepsilon}{2t+\varepsilon-2}} \right)^{2t+\varepsilon-2},
\end{align}
where $\calM^{\calD}=\calM_1^{\calD}$ is the standard maximal operator over dyadic Carleson tents, and in the last estimate, we have used the fact that $\{Q_{\alpha, i}\}_{i \ge 1}$ are mutually disjoint.

\vspace{0.1cm}

Now for any $f \in L^{\frac{1}{3-2t-\varepsilon}}(\D)$, we have to estimate $\left\|\calM_t^{\calD}f \right\|_{L^1(\D)}$. Write
\begin{equation} \label{20260115eq01}
\int_{\D}\calM_t^{\calD}f(z)  dA(z)
=\sum_{\ell \in \Z} \int_{E_{4^{t\ell}} \backslash E_{4^{t(\ell+1)}}}   \calM_t^{\calD}f(z)dA(z) \lesssim \sum_{\ell \in \Z} 4^{t\ell}\left|E_{4^{t\ell}} \backslash E_{4^{t(\ell+1)}} \right|.
\end{equation} 
Observe that it can happen that the intersection of the two sets of Carleson boxes $\{Q_{4^{t\ell}, i}\}_{i \ge 1}$ and $\{Q_{4^{t(\ell+1)}, i}\}_{i \ge 1}$ is not empty. Hence, for each $\ell \in \Z$, define the disjoint union of Carleson boxes
\begin{equation} \label{20260115eq30}
\{\widetilde{Q}_{4^{t\ell}, i}\}_{i \ge 1}:=\{Q_{4^{t\ell}, i}\}_{i \ge 1} \backslash \{Q_{4^{t(\ell+1)}, i}\}_{i \ge 1}.
\end{equation} 
Note that 
\begin{enumerate} 
\item [$\bullet$] $E_{4^{t\ell}} \backslash E_{4^{t(\ell+1)}} \subseteq \bigcup_{i \ge 1}\widetilde{Q}_{4^{t\ell}, i} $.
\item [$\bullet$] $\{\widetilde{Q}^{\textrm{up}}_{4^{t\ell}, i}\}_{\ell \in \Z, \; i \ge 1}$ are mutually disjoint. 
\end{enumerate} 
Using now \eqref{20251218eq01} with $E_{\alpha}$ replaced by $E_{4^{t\ell}} \backslash E_{4^{t(\ell+1)}}$, and $\{Q_{\alpha, i}\}_{i \ge 1}$ replaced by $\{\widetilde{Q}_{4^{t\ell}, i}\}_{i \ge 1}$, respectively, we have 
\begin{align} \label{20251221eq02}
\textrm{RHS of \eqref{20260115eq01}} 
& \le \sum_{\ell \in \Z} \left(\sum_{i=1}^\infty \int_{\widetilde{Q}^{\textnormal{up}}_{4^{t\ell}, i}} \left|\calM^{\calD} f(z) \right|^{\frac{1}{3-2t-\varepsilon}} dA(z)  \right)^{3-2t-\varepsilon} \cdot \left( \sum_{i=1}^\infty |\widetilde{Q}_{4^{t\ell}, i}|^{\frac{t-1+\varepsilon}{2t+\varepsilon-2}} \right)^{2t+\varepsilon-2} \nonumber \\
& \le \left(\sum_{\ell \in \Z } \sum_{i=1}^\infty \int_{\widetilde{Q}^{\textnormal{up}}_{4^{t\ell}, i}} \left|\calM^{\calD} f(z) \right|^{\frac{1}{3-2t-\varepsilon}} dA(z)  \right)^{3-2t-\varepsilon} \cdot \left(\sum_{\ell
 \in \Z} \sum_{i=1}^\infty |\widetilde{Q}_{4^{t\ell}, i}|^{\frac{t-1+\varepsilon}{2t+\varepsilon-2}} \right)^{2t+\varepsilon-2}.
\end{align}
Since $\{\widetilde{Q}^{\textnormal{up}}_{4^{t\ell}, i} \}_{\ell \in \Z, \; i \ge 1}$ are mutually disjoint, therefore, the first double sum in \eqref{20251221eq02} is bounded above by 
\begin{equation} \label{20251221eq03}
 \int_{\D}\left|\calM^{\calD} f(z) \right|^{\frac{1}{3-2t-\varepsilon}} dA(z)=\left\| \calM^{\calD} f \right\|^{\frac{1}{3-2t-\varepsilon}}_{L^{\frac{1}{3-2t-\varepsilon}}(\D)} \lesssim \left\|  f \right\|^{\frac{1}{3-2t-\varepsilon}}_{L^{\frac{1}{3-2t-\varepsilon}}(\D)},
\end{equation}
where in the last estimate above, we have used the boundedness of $\calM^{\calD}: L^{\frac{1}{3-2t-\varepsilon}}(\D) \to L^{\frac{1}{3-2t-\varepsilon}}(\D)$. 

We are left with estimating the second double summation in \eqref{20251221eq02}. Indeed, we have 
\begin{align*}
\sum_{\ell 
 \in \Z} \sum_{i=1}^\infty |\widetilde{Q}_{4^{t\ell}, i}|^{\frac{t-1+\varepsilon}{2t+\varepsilon-2}}&
\simeq \sum_{\ell 
 \in \Z} \sum_{i=1}^\infty |\widetilde{Q}^{\textnormal{up}}_{4^{t\ell}, i}|^{\frac{t-1+\varepsilon}{2t+\varepsilon-2}}=\sum_{\ell \in \Z} \sum_{i=1}^\infty \left|\widetilde{Q}^{\textnormal{up}}_{4^{t\ell}, i} \right| \cdot \frac{1}{  \left|\widetilde{Q}^{\textnormal{up}}_{4^{t\ell}, i} \right|^{\frac{t-1}{2t+\varepsilon-2}}} \\
&\simeq \sum_{\ell \in \Z} \sum_{i=1}^\infty \int_{\widetilde{Q}^{\textnormal{up}}_{4^{t\ell}, i}} \frac{1}{(1-|z|^2)^{\frac{2(t-1)}{2t+\varepsilon-2}}} dA(z) \lesssim \int_{\D} \frac{1}{(1-|z|^2)^{\frac{2(t-1)}{2t+\varepsilon-2}}} dA(z)\\
& \simeq \int_0^1 \frac{1}{(1-r)^{\frac{2(t-1)}{2t+\varepsilon-2}}} dr<+\infty. 
\end{align*}
The desired claim follows by combining the above estimate with \eqref{20251221eq02} and \eqref{20251221eq03}.
\end{proof}

Therefore, we derive the following result. 

\begin{prop} \label{20251228prop01}
For any $(p, q)$ belonging to the off-critical line regime, namely,
$$
(p, q) \in \left\{ \left(\frac{1}{p}, \frac{1}{q} \right) \in [0, 1]^2: \frac{1}{q}-\frac{1}{p}>2t-2 \right\},
$$
one has $\calM_t^{\calD}: L^p(\D) \to L^q(\D)$ is bounded.
\end{prop}

\begin{proof}
The desired claim follows clearly by interpolating the estimates derived in Observation \ref{20251219obs01} and Lemma \ref{20251219lem01}.
\end{proof}

\subsection{Critical line estimate}

We first deal with the endpoint $\left(1/p, 1/q \right)=(0, 2t-2)$.

\begin{lem} \label{20251219lem02}
$\calM_t^{\calD}: L^\infty(\D) \to L^{\frac{1}{2t-2}, \infty}(\D)$ is bounded.
\end{lem}

\begin{proof}
Take any $\alpha>0$ and any measurable function $f$ with $\left\|f \right\|_{L^\infty(\D)}=1$. Since $\left| \calM_t^{\calD}f(z) \right|>\alpha$, we have 
$$
\alpha\le \sup_{z \in Q_I, \; I \in \calD} \frac{1}{|Q_I|^{t-1}} \simeq \frac{1}{(1-|z|^2)^{2(t-1)}}. 
$$
Therefore, 
\begin{align*}
\left| \left\{z \in \D: \; \left| \calM_t^{\calD}f(z) \right|>\alpha \right\} \right| 
& \le \left| \left\{z \in \D: \sup_{z \in Q_I, \; I \in \calD} \frac{1}{|Q_I|}>\alpha^{\frac{1}{t-1}} \right\} \right| \\
& \lesssim \left| \left\{z \in \D: \; \frac{1}{(1-|z|^2)^2}>\alpha^{\frac{1}{t-1}} \right\} \right| \\
&= \left| \left\{ z \in \D: 1-\alpha^{-\frac{1}{2(t-1)}} \le |z|^2 < 1\right\} \right| \\
& \simeq \alpha^{-\frac{1}{2(t-1)}},
\end{align*}
which gives
$$
\alpha \left| \left\{z \in \D: \; \left| \calM_t^{\calD}f(z) \right|>\alpha \right\} \right|^{2t-2} \lesssim 1. 
$$
The proof is complete. 
\end{proof}

Next, we treat the other endpoint $\left(\frac{1}{p}, \frac{1}{q} \right)=(3-2t, 1)$. 

\begin{thm} \label{20251222thm01}
Let $\omega$ be a weight on $[0,1)$ satisfying $\omega(r)\ge c>0$ for all $r\in[0,1/2)$, and let $\omega(z):=\omega(|z|)$ denote the associated radial weight on $\D$.
 Then $\calM_t^{\calD}: L^{\frac{1}{3-2t}}(\D, \omega) \to L^{1, \infty}(\D)$ is bounded if and only if
\begin{equation} \label{20251220eq13}
\sup_{k \ge 0} 2^k \int_{1-\frac{1}{2^k}}^{1-\frac{1}{2^{k+1}}} \frac{1}{\omega^{\frac{3-2t}{2t-2}}(r)}dr<+\infty. 
\end{equation} 
\end{thm}

\begin{proof}
{\bf Sufficiency.} Let $\alpha>0$ and denote $E:=\left\{z \in \D: \calM_t^{\calD}f(z)>\alpha \right\}$. As usual, we decompose $E$ into a union of maximal and mutually disjoint Carleson boxes $\{Q_i\}_{i \ge 1}$ satisfying $Q_i=Q_{I_i}$ for $I_i \in \calD$ and 
$$
\frac{1}{|Q_i|^t} \int_{Q_i} |f(z)|dA(z)>\alpha \qquad \textnormal{for each} \; \; i \ge 1.
$$
This gives
\begin{align} \label{20251220eq14}
 \alpha |E|&
 =\alpha \sum_{i=1}^\infty |Q_i| \le \sum_{i=1}^\infty \frac{1}{|Q_i|^{t-1}} \int_{Q_i} |f(z)|dA(z) \nonumber \\
 &=\sum_{i=1}^\infty \frac{1}{|Q_i|^{t-1}} \left( \sum_{J \subseteq I_i, \; J \; \textnormal{dyadic}}  \int_{Q_J^{\textnormal{up}}} |f(z)| \omega^{3-2t}(z) \cdot \frac{1}{\omega^{3-2t}(z)} dA(z) \right) \nonumber \\
 & \le \sum_{i=1}^\infty \frac{1}{|Q_i|^{t-1}} \sum_{J \subseteq I_i, \; J \; \textnormal{dyadic}} \left(\int_{Q_J^{\textnormal{up}}}|f(z)|^{\frac{1}{3-2t}} \omega(z)dA(z) \right)^{3-2t} \left(\int_{Q_J^{\textnormal{up}}} \frac{dA(z)}{\omega^{\frac{3-2t}{2t-2}}(z)} \right)^{2t-2} \nonumber \\
 & \le \left(\sum_{i=1}^\infty \sum_{J \subseteq I_i, \; J \; \textnormal{dyadic}} \int_{Q_J^{\textnormal{up}}} |f(z)|^{\frac{1}{3-2t}} \omega(z) dA(z) \right)^{3-2t} \left( \sum_{i=1}^\infty \sum_{J \subseteq I_i, \; J \; \textrm{dyadic}} \frac{1}{|Q_i|^{\frac{1}{2}}} \int_{Q_J^{\textnormal{up}}} \frac{dA(z)}{\omega^{\frac{3-2t}{2t-2}}(z)} \right)^{2t-2} \nonumber \\
 &\le \left\|f \right\|_{L^{\frac{1}{3-2t}}(\D, \omega)}  \cdot \left( \sum_{i=1}^\infty \sum_{J \subseteq I_i, \; J \; \textrm{dyadic}} \frac{1}{|Q_i|^{\frac{1}{2}}} \int_{Q_J^{\textnormal{up}}} \frac{dA(z)}{\omega^{\frac{3-2t}{2t-2}}(z)} \right)^{2t-2} 
\end{align}
For the double summation in the above estimate, we have\footnote{Here and henceforth, we normalize the length of $\T$ so that $|\T|=1$. In particular, any dyadic descendant of $\T$ has length $2^{-k}$ for some $k \ge 0$.}
\begin{align} \label{20251220eq11}
    \sum_{i=1}^\infty \sum_{J \subseteq I_i, \; J \; \textrm{dyadic}} \frac{1}{|Q_i|^{\frac{1}{2}}} \int_{Q_J^{\textnormal{up}}} \frac{dA(z)}{\omega^{\frac{3-2t}{2t-2}}(z)} 
    &\simeq \sum_{i=1}^\infty \sum_{k \ge 0} \sum_{\substack{J \subseteq I_i, \; J \; \textrm{dyadic} \\ |J|=2^{-k}|I_i|}} \frac{1}{2^k |Q_J^{\textnormal{up}}|^{\frac{1}{2}}} \int_{Q_J^{\textnormal{up}}} \frac{dA(z)}{\omega^{\frac{3-2t}{2t-2}}(z)} \nonumber \\
    & \simeq \sum_{k \ge 0} \frac{1}{2^k} \left( \sum_{i=1}^\infty  \sum_{\substack{J \subseteq I_i, \; J \; \textrm{dyadic} \\ |J|=2^{-k}|I_i|}} \int_{Q_J^{\textnormal{up}}} \frac{dA(z)}{(1-|z|^2)\omega^{\frac{3-2t}{2t-2}}(z)} \right).
\end{align}
Using now the assumption that $\omega$ is radial and \eqref{20251220eq13}, we have 
\begin{align*}
\int_{Q_J^{\textnormal{up}}} \frac{dA(z)}{(1-|z|^2)\omega^{\frac{3-2t}{2t-2}}(z)} 
& \simeq \frac{1}{|J|} \int_{Q_J^{\textnormal{up}}} \frac{dA(z)}{\omega^{\frac{3-2t}{2t-2}}(z)} \lesssim |J| \cdot  \frac{1}{|J|} \int_{1-|J|}^{1-\frac{|J|}{2}} \frac{dr}{\omega^{\frac{3-2t}{2t-2}}(r)} \lesssim |J|.
\end{align*}
Substituting the above estimate back to \eqref{20251220eq11}, we have 
\begin{align*}
    \textnormal{RHS of \eqref{20251220eq11}}
    & \lesssim \sum_{k \ge 0} \frac{1}{2^k} \left( \sum_{i=1}^\infty  \sum_{\substack{J \subseteq I_i, \; J \; \textrm{dyadic} \\ |J|=2^{-k}|I_i|}} |J| \right) \\
    &=\sum_{k \ge 0} \frac{1}{2^k} \left(\sum_{i=1}^{\infty} |I_i| \right)<+\infty,
\end{align*}
where in the last estimate we have used the fact that $I_i$'s are mutually disjoint. The desired weak-type estimate $\calM_t^{\calD}: L^{\frac{1}{3-2t}}(\D, \omega) \to L^{1, \infty}(\D)$ then follows by plugging the above estimate back into \eqref{20251220eq14}. 

\vspace{0.1cm}

{\bf Necessity.} For each $k \ge 0$, denote 
\begin{equation} \label{20251222eq55}
D_k:=\left\{z \in \D: 1-\frac{1}{2^k} \le |z| <1-\frac{1}{2^{k+1}} \right\},
\end{equation} 
and $f_k(z):=\omega^{-\frac{3-2t}{2t-2}}(z) \one_{D_k}(z)$. A direct computation yields
\begin{equation} \label{20251221eq01}
\left\|f_k \right\|_{L^{\frac{1}{3-2t}}(\D, \omega)}=\left(\int_{D_k} \omega^{-\frac{3-2t}{2t-2}}(z)dA(z) \right)^{3-2t} \simeq \left( \int_{1-\frac{1}{2^k}}^{1-\frac{1}{2^{k+1}}} \frac{dr}{\omega^{\frac{3-2t}{2t-2}}(r)} \right)^{3-2t}.
\end{equation} 
Next, consider $\calD_k:=\left\{I \in \calD: |I|=2^{-k} \right\}$. Then for each $w \in Q_I$ with $I \in \calD_k$, one has
\begin{align*}
\calM_t^{\calD}f_k(w)
&\ge \frac{1}{|Q_I|^t} \int_{Q_I} f_k(z)dA(z) \ge \frac{1}{|Q_I|^t} \int_{Q_I^{\textnormal{up}}}\frac{1}{\omega^{\frac{3-2t}{2t-2}}(z)}dA(z) \\
& \simeq \frac{1}{|Q_I|^t} \cdot |I| \cdot \int^{1-\frac{|I|}{2}}_{1-|I|} \frac{dr}{\omega^{\frac{3-2t}{2t-2}}(r)} \simeq 2^{k(2t-1)} \cdot \int_{1-\frac{1}{2^k}}^{1-\frac{1}{2^{k+1}}} \frac{dr}{\omega^{\frac{3-2t}{2t-2}}(r)},
\end{align*}
which implies
$$
Q_I \subseteq  \left\{w \in \D: \calM_t^{\calD}f_k(w)>\widetilde{c}2^{k(2t-1)} \cdot \int_{1-\frac{1}{2^k}}^{1-\frac{1}{2^{k+1}}} \frac{dr}{\omega^{\frac{3-2t}{2t-2}}(r)} \right\}
$$
for some absolute constant $\widetilde{c}>0$ being sufficiently small. Since $I \in \calD_k$ are mutually disjoint, we further get
\begin{equation} \label{20251222eq60}
\bigcup_{I \in \calD_k} Q_I \subseteq  \left\{w \in \D: \calM_t^{\calD}f_k(w)>\widetilde{c}2^{k(2t-1)} \cdot \int_{1-\frac{1}{2^k}}^{1-\frac{1}{2^{k+1}}} \frac{dr}{\omega^{\frac{3-2t}{2t-2}}(r)} \right\},
\end{equation} 
Therefore, by the assumption $\calM_t^{\calD}: L^{\frac{1}{3-2t}}(\D, \omega) \to L^{1, \infty}(\D)$ and \eqref{20251221eq01}, we have 
\begin{align*}
2^{-k} & \simeq \left| \bigcup_{I \in \calD_k} Q_I \right| \lesssim \left|  \left\{w \in \D: \calM_t^{\calD}f_k(w)>\widetilde{c}2^{k(2t-1)} \cdot \int_{1-\frac{1}{2^k}}^{1-\frac{1}{2^{k+1}}} \frac{dr}{\omega^{\frac{3-2t}{2t-2}}(r)} \right\} \right| \\
& \lesssim \frac{1}{2^{k(2t-1)} \cdot \int_{1-\frac{1}{2^k}}^{1-\frac{1}{2^{k+1}}} \frac{dr}{\omega^{\frac{3-2t}{2t-2}}(r)}} \cdot \left\|f_k \right\|_{L^{\frac{1}{3-2t}}(\D, \omega)}  \\
& \lesssim \frac{1}{2^{k(2t-1)} \cdot \int_{1-\frac{1}{2^k}}^{1-\frac{1}{2^{k+1}}} \frac{dr}{\omega^{\frac{3-2t}{2t-2}}(r)}}  \left( \int_{1-\frac{1}{2^k}}^{1-\frac{1}{2^{k+1}}} \frac{dr}{\omega^{\frac{3-2t}{2t-2}}(r)} \right)^{3-2t} \\
&=\frac{1}{2^{k(2t-1)}} \cdot \left( \int_{1-\frac{1}{2^k}}^{1-\frac{1}{2^{k+1}}} \frac{dr}{\omega^{\frac{3-2t}{2t-2}}(r)} \right)^{2-2t},
\end{align*}
which gives 
$$
2^k \int_{1-\frac{1}{2^k}}^{1-\frac{1}{2^{k+1}}} \frac{dr}{\omega^{\frac{3-2t}{2t-2}}(r)} \lesssim 1.
$$
The proof of the necessity is complete. 
\end{proof}

\begin{rem}
In Theorem \ref{20251222thm01}, the assumption $\omega(r)\ge c>0$ for
$r\in\left[0,\frac12\right)$ is used only to ensure that $
\int_{0}^{1/2}\omega(r)^{-\frac{3-2t}{2t-2}}\,dr<\infty$. This is a minor technical requirement, since the relevant (and more delicate) behavior of the weight occurs near the boundary, as $r\to 1^{-}$.
\end{rem}

As a direct application of Theorem \ref{20251222thm01} with $\omega \equiv 1$, we have $\calM_t^{\calD}: L^{\frac{1}{3-2t}}(\D) \to L^{1, \infty}$ is bounded. Interpolating this with Lemma \ref{20251219lem02}, we derive the following estimates on the critical line.

\begin{cor} \label{20251226lem10}
For any $(p, q)$ belonging to the critical line regime, namely,
$$
 (p, q) \in \left\{ \left(\frac{1}{p}, \frac{1}{q} \right) \in (0, 1)^2: \frac{1}{q}-\frac{1}{p}=2t-2 \right\},
$$
then for every $0<r<\infty$, the operator $\calM_t^{\calD}$ extends a bounded operator from $L^{p,r}(\D)$ to $L^{q,r}(\D)$. In particular, $\calM_t^{\calD}: L^p(\D) \to L^{q, \infty}(\D)$ is bounded.  
\end{cor}

\begin{proof}
The first assertion follows from the off-diagonal Marcinkiewicz interpolation theorem \cite[Theorem~1.4.19]{Grafakos2014}. The second follows from by letting $r=\infty$ and the fact that $L^p(\D) \subset L^{p, \infty}(\D)$. 
\end{proof}
A natural question arising from Theorem~\ref{20251222thm01} is whether $\calM_t^{\calD}: L^{\frac{1}{3-2t}}(\D)\to L^1(\D)$ is bounded or not. Here, again, we would like to formulate this result in the setting of radial weights. For this purpose, we recall the notion of \emph{B\'ekoll\'e--Bonami weights}. Let \(1<l<\infty\). We say that a weight \(\omega\) on \(\D\) belongs to the \emph{B\'ekoll\'e--Bonami class \({\bf B}_l\)} if
\begin{equation} \label{20251228defn21}
[\omega]_{{\bf B}_l}
:=\sup_{I\subseteq\T}
\left(\frac{1}{|Q_I|}\int_{Q_I}\omega\, dA\right)
\left(\frac{1}{|Q_I|}\int_{Q_I}\omega^{-\frac{1}{l-1}}\, dA\right)^{l-1}
<\infty,
\end{equation} 
In particular, if \(\omega\) is radial, then the above condition is equivalent to
\[
[\omega]_{{\bf B}_l}
 \simeq \sup_{0<h<1}
\left(\frac{1}{h}\int_{1-h}^1 \omega(r)\,dr\right)
\left(\frac{1}{h}\int_{1-h}^1 \omega(r)^{-\frac{1}{l-1}}\,dr\right)^{l-1}<+\infty
\]

An important reason to consider the B\'ekoll\'e--Bonami weights is that they provide a necessary and sufficient condition for 
\begin{enumerate} 
\item [$\bullet$] the standard Hardy--Littlewood maximal operator over all Carleson tents \(\calM\) to extend to a bounded operator on \(L^l(\D,\omega)\), with operator norm of magnitude $[\omega]_{{\bf B}_l}^{\frac{1}{l-1}}$ (see, e.g.,  \cite{APM2019, PR2013}). 
\item [$\bullet$] the Bergman projection $\mathcal P$ acting as a bounded operator on $L^l(\D, \omega)$ (see, e.g.,  \cite{BB1978, PR2013, RTW17}). 
\end{enumerate}

\begin{thm} \label{20251222thm02}
Let $\omega$ be a radial weight that satisfies the assumption of Theorem \ref{20251222thm01}. Let further, $\omega \in {\bf B}_{\frac{1}{3-2t}}$. Then $\calM_t^{\calD}: L^{\frac{1}{3-2t}}(\D, \omega) \to L^1(\D)$ is bounded if and only if 
$$
\sum_{k \ge 0} 2^k \int_{1-\frac{1}{2^k}}^{1-\frac{1}{2^{k+1}}} \frac{1}{\omega^{\frac{3-2t}{2t-2}}(r)}dr<+\infty.
$$
In particular, $\calM_t^{\calD}$ maps $L^{\frac{1}{3-2t}}(\D)$ unboundedly into $L^1(\D)$. 
\end{thm}

\begin{proof}
We prove Theorem \ref{20251222thm02} by adapting the ideas from the proof of Lemma \ref{20251219lem01} and Theorem \ref{20251222thm01}. 

\vspace{0.1cm}

\noindent {\bf Sufficiency.} Again, for any $\alpha>0$, we denote $E_\alpha:=\left\{z \in \D: \calM_t^{\calD}f(z)>\alpha\right\}$. Then as usual, one can decompose $E_\alpha$ into disjoint union of maximal Carleson boxes $\{Q_{\alpha, i}\}_{i \ge 1}$ enjoying \eqref{20251222eq50}. Then following the argument in \eqref{20251218eq01} with applying H\"older for the conjugate pair $(l, l')=\left(\frac{1}{3-2t}, \frac{1}{2t-2} \right)$, we have 
\begin{align} \label{20260115eq43}
\alpha |E_\alpha| &
\lesssim \sum_{i=1}^\infty \frac{1}{|Q_{\alpha, i}|^{t-1}} \int_{Q_{\alpha, i}^{\textnormal{up}}} \calM^{\calD}f(z) \omega^{3-2t}(z) \cdot  \frac{1}{\omega^{3-2t}(z)} dA(z)  \nonumber \\ 
&\lesssim \left(\sum_{i=1}^\infty \int_{Q_{\alpha, i}^{\textnormal{up}}} \left| \calM^{\calD} f(z) \right|^{\frac{1}{3-2t}} \omega(z) dA(z) \right)^{3-2t} \left(\sum_{i=1}^\infty \int_{Q_{\alpha, i}^{\textnormal{up}}} \frac{dA(z)}{(1-|z|^2) \omega^{\frac{3-2t}{2t-2}}(z)} \right)^{2t-2}. 
\end{align}
To estimate $\left\|\calM_t^{\calD} f \right\|_{L^1(\D)}$, write 
\begin{equation} \label{20260115eq40}
\int_{\D} \calM_t^{\calD}f(z) dA(z)
=\sum_{\ell \in \Z} \int_{E_{4^{t\ell}} \backslash E_{4^{t(\ell+1)}}} \calM_t^{\calD}f(z)dA(z) \lesssim \sum_{\ell \in \Z} 4^{t\ell} \left|E_{4^{t\ell}} \backslash E_{4^{t(\ell+1)}} \right|.
\end{equation} 
Similarly as in \eqref{20260115eq30}, let
$$
\left\{\widetilde{Q}_{4^{t\ell}, i} \right\}_{i \ge 1}:=\left\{Q_{4^{t\ell}, i} \right\}_{i \ge 1} \backslash \left\{Q_{4^{t(\ell+1)}, i} \right\}_{i \ge 1}.
$$
Then using \eqref{20260115eq43} with $E_{\alpha}$ replaced by $E_{4^{t\ell}} \backslash E_{4^{t(\ell+1)}}$, and $\{Q_{\alpha, i}\}_{i \ge 1}$ replaced by $\{\widetilde{Q}_{4^{t\ell}, i}\}_{i \ge 1}$, respectively, we have  
\begin{align} \label{20251222eq53}
& \textrm{RHS of \eqref{20260115eq40}} \nonumber \\
& \le \sum_{\ell \in \Z} \left(\sum_{i=1}^\infty \int_{\widetilde{Q}_{4^{t\ell} i}^{\textnormal{up}}} \left| \calM^{\calD} f(z) \right|^{\frac{1}{3-2t}} \omega(z) dA(z) \right)^{3-2t} \left(\sum_{i=1}^\infty \int_{\widetilde{Q}_{4^{t\ell}, i}^{\textnormal{up}}} \frac{dA(z)}{(1-|z|^2) \omega^{\frac{3-2t}{2t-2}}(z)} \right)^{2t-2} \nonumber\\
& \lesssim \left\|\calM^{\calD}f \right\|_{L^{\frac{1}{3-2t}}(\D, \; \omega)} \cdot \left(\sum_{\ell \in \Z} \sum_{i=1}^\infty \int_{\widetilde{Q}_{4^{t\ell}, i}^{\textnormal{up}}} \frac{dA(z)}{(1-|z|^2) \omega^{\frac{3-2t}{2t-2}}(z)} \right)^{2t-2} \nonumber \\
& \lesssim [\omega]_{{\bf B}_{\frac{1}{3-2t}}}^{\frac{3-2t}{2t-2}} \cdot \left\|f \right\|_{L^{\frac{1}{3-2t}}(\D, \; \omega)} \cdot \left(\sum_{\ell \in \Z} \sum_{i=1}^\infty \int_{\widetilde{Q}_{4^{t\ell}, i}^{\textnormal{up}}} \frac{dA(z)}{(1-|z|^2) \omega^{\frac{3-2t}{2t-2}}(z)} \right)^{2t-2},
\end{align}
where in the last estimate, we used the assumption $\omega \in {\bf B}_{\frac{1}{3-2t}}$. Finally for the double summation in \eqref{20251222eq53}, using the fact that $\{\widetilde{Q}_{4^{t\ell}, i}^{\textnormal{up}}\}_{\ell \in \Z, \; i \ge 1}$ are mutually disjoint, we have 
\begin{align*}
\sum_{\ell \in \Z} \sum_{i=1}^\infty \int_{\widetilde{Q}_{4^{t\ell}, i}^{\textnormal{up}}} \frac{dA(z)}{(1-|z|^2) \omega^{\frac{3-2t}{2t-2}}(z)}
& \lesssim \int_{\D} \frac{1}{(1-|z|^2)\omega^{\frac{3-2t}{2t-2}}(z)} dA(z)\\
& \simeq \int_0^1 \frac{rdr}{(1-r^2)\omega^{\frac{3-2t}{2t-2}}(r)} \\ 
&\simeq \sum_{k \ge 0} 2^k \int_{1-\frac{1}{2^k}}^{1-\frac{1}{2^{k+1}}} \frac{1}{\omega^{\frac{3-2t}{2t-2}}(r)}dr<+\infty. 
\end{align*}
The sufficiency is therefore proved. 

\vspace{0.1cm}

\noindent {\bf Necessity.} Assume $\calM_t^{\calD}: L^{\frac{1}{3-2t}}(\D, \omega) \to L^1(\D)$ is bounded. Let $N \in \N$ be sufficiently large, and take the test function 
$$
f_N(z):=\sum_{k=0}^N 2^{k(3-2t)} \omega^{-\frac{3t-2}{2t-2}}(z) \one_{D_k}(z), 
$$
where $D_k$ is the annulus defined as in \eqref{20251222eq55}. 

On one hand side, recall that $\{D_k\}_{k \ge 0}$ are mutually disjoint, we have 
\begin{align} \label{20251222eq70}
\left\|f_N \right\|_{L^{\frac{1}{3-2t}}(\D, \omega)}
&=\left( \int_\D \left| \sum_{k=0}^N 2^{k(3-2t)} \omega^{-\frac{3t-2}{2t-2}}(z) \one_{D_k}(z) \right|^{\frac{1}{3-2t}} \omega(z) dA(z) \right)^{3-2t} \nonumber \\
&=\left( \sum_{k' \ge 0} \int_{D_{k'}} \left| \sum_{k=0}^N  2^{k(3-2t)} \omega^{-\frac{3t-2}{2t-2}}(z) \one_{D_k}(z) \right|^{\frac{1}{3-2t}} \omega(z) dA(z) \right)^{3-2t} \nonumber \\
&=\left( \sum_{k=0}^N \int_{D_k} \left|  2^{k(3-2t)} \omega^{-\frac{3t-2}{2t-2}}(z) \one_{D_k}(z) \right|^{\frac{1}{3-2t}} \omega(z) dA(z) \right)^{3-2t}  \nonumber \\
&= \left(\sum_{k=0}^N  2^k \int_{D_k} \omega^{-\frac{3-2t}{2t-2}}(z)dA(z) \right)^{3-2t} \simeq \left(\sum_{k=0}^N 2^k \int_{1-\frac{1}{2^k}}^{1-\frac{1}{2^{k+1}}} \frac{dr}{\omega^{\frac{3-2t}{2t-2}}(r)} \right)^{3-2t}.
\end{align} 
On the other hand side, by the argument in \eqref{20251222eq60}, we see that for $0 \le k \le N$,  
$$
\calM_t^{\calD} f_N(z) \gtrsim 2^{k(3-2t)} \cdot 2^{k(2t-1)} \cdot \int_{1-\frac{1}{2^k}}^{1-\frac{1}{2^{k+1}}} \frac{dr}{\omega^{\frac{3-2t}{2t-2}}(r)}=2^{2k} \int_{1-\frac{1}{2^k}}^{1-\frac{1}{2^{k+1}}} \frac{dr}{\omega^{\frac{3-2t}{2t-2}}(r)}, 
$$
and hence 
\begin{align} \label{20251222eq71}
\left\|\calM_t^{\calD} f_N \right\|_{L^1(\D)}
& \ge \sum_{k=0}^N \int_{D_k} \calM_t^{\calD} f_N(z)dA(z) \nonumber \\
& \gtrsim \sum_{k=0}^N 2^{-k} \cdot 2^{2k} \int_{1-\frac{1}{2^k}}^{1-\frac{1}{2^{k+1}}} \frac{dr}{\omega^{\frac{3-2t}{2t-2}}(r)} \nonumber \\
&= \sum_{k=0}^N 2^{k} \int_{1-\frac{1}{2^k}}^{1-\frac{1}{2^{k+1}}} \frac{dr}{\omega^{\frac{3-2t}{2t-2}}(r)}.
\end{align}
Finally, since $\left\|\calM_t^{\calD} f_N\right\|_{L^1(\D)} \lesssim \left\|f_N \right\|_{L^{\frac{1}{3-2t}}(\D, \omega)}$, using this with \eqref{20251222eq70} and \eqref{20251222eq71}, we derive that 
\begin{equation} \label{20251222eq73}
\sum_{k=0}^N 2^{k} \int_{1-\frac{1}{2^k}}^{1-\frac{1}{2^{k+1}}} \frac{dr}{\omega^{\frac{3-2t}{2t-2}}(r)} \lesssim \left(\sum_{k=0}^N 2^{k} \int_{1-\frac{1}{2^k}}^{1-\frac{1}{2^{k+1}}} \frac{dr}{\omega^{\frac{3-2t}{2t-2}}(r)} \right)^{3-2t}. 
\end{equation} 
Note that since $\omega$ is a weight (hence locally integrable), this means
$$
\sum_{k=0}^N 2^{k} \int_{1-\frac{1}{2^k}}^{1-\frac{1}{2^{k+1}}} \frac{dr}{\omega^{\frac{3-2t}{2t-2}}(r)} \simeq \int_0^{1-\frac{1}{2^{N+1}}} \frac{dr}{(1-r)\omega^{\frac{3-2t}{2t-2}}(r)}<+\infty, 
$$
and hence \eqref{20251222eq73} gives
$$
\left(\sum_{k=0}^N 2^{k} \int_{1-\frac{1}{2^k}}^{1-\frac{1}{2^{k+1}}} \frac{dr}{\omega^{\frac{3-2t}{2t-2}}(r)} \right)^{2t-2} \lesssim 1, 
$$
where the implicit constant in the above estimate is \emph{independent} of the choice of $N$. Finally, the desired necessary part follows by letting $N \to \infty$ in the above estimate and the assumption that $1<t<3/2$. 
\end{proof}

To this end, we show that the $L^p$ estimates we derived for $\calM_t^{\calD}$ are \emph{sharp}. 

\begin{prop} \label{20260131prop01}
For any $1 \le p, q \le +\infty$ with $\frac{1}{q}-\frac{1}{p}<2t-2$, $\calM_t^{\calD}$ maps $L^p(\D)$ unboundedly to $L^{q, \infty}(\D)$.
\end{prop}

\begin{proof}
To begin with, we may assume that \(p,q<+\infty\). The case \(p=+\infty\) follows by an argument similar to that in Observation \ref{20251219obs01}, while the case \(q=+\infty\) is trivial.

Take $\eps>0$ sufficiently small, such that $1/q-1/(p+\eps)<2t-2$, and let the test function $f(z):=(1-|z|^2)^{-\frac{1}{p+\eps}}, \; z \in \D$. On one hand, 
$$
\left\|f \right\|_{L^p(\D)}^p \simeq \int_0^1 \frac{dr}{(1-r)^{\frac{p}{p+\eps}}}<+\infty.
$$
On the other hand, 
\begin{align*}
\calM_t^{\calD} f(z)
&= \sup_{z \in Q_I, \; I \in \calD} \frac{1}{|Q_I|^t} \int_{Q_I} \frac{1}{(1-|w|^2)^{\frac{1}{p+\eps}}} dA(w)  \\
& \ge \sup_{z \in Q_I, \; I \in \calD} \frac{1}{|Q_I|^t} \int_{Q_I^{\textrm{up}}} \frac{1}{(1-|w|^2)^{\frac{1}{p+\eps}}} dA(w)  \\
& \simeq \sup_{z \in Q_I, \; I \in \calD} \frac{1}{|Q_I|^{t-1+\frac{1}{2(p+\eps)}}}  \simeq \frac{1}{(1-|z|^2)^{2(t-1)+\frac{1}{p+\eps}}}, 
\end{align*}
which gives 
$$
\left| \left\{z \in \D: \calM_t^{\calD}f(z) \gtrsim \lambda \right\} \right| \gtrsim \min \left\{1, \lambda^{-\frac{1}{q\left(2t-2+\frac{1}{p+\eps}\right)}} \right\}.
$$
Hence,
$$
\left\|\calM_t^{\calD} f \right\|_{L^{q, \infty}(\D)}=\sup_{\lambda>0} \lambda \left| \left\{z \in \D: \calM_t^{\calD}f(z) \gtrsim \lambda \right\} \right|^{\frac{1}{q}} \gtrsim \sup_{\lambda>0} \lambda \cdot \min \left\{1, \lambda^{-\frac{1}{q\left(2t-2+\frac{1}{p+\eps}\right)}} \right\}=+\infty,
$$
where in the last estimate above, we used the fact that $q \left(2t-2+\frac{1}{p+\eps} \right)>1$.
\end{proof}

\subsection{A two-weight estimate for $\mathcal{M}_t^{\mathcal{D}}$}

The last part of this section concerns weighted estimates for $\mathcal{M}_t^{\mathcal{D}}, \; t>0$ in the (hyper-singular) regime $\{1 \le p, q \le +\infty: p>q\}$. We start with some definition. 

\begin{defn} \label{20260119defn01}
Let $\eta$ be a weight on $\D$ and $\calD$ be a dyadic system on $\T$. We say $\eta$ belongs to the \emph{dyadic ${\bf B}_\infty$ 
class} ${\bf B}_{\infty}(\calD)$ if there exists an absolute constant $C>1$, such that for each $I \in \calD$, 
$$
\eta(Q_I) \le C \eta(Q_I^{\textrm{up}}). 
$$
\end{defn}

\begin{rem}
The ${\bf B}_\infty(\mathcal D)$ condition is a rather mild assumption, and it includes most of the well--known examples of weights on $\mathbb D$. Here are some examples. 
\begin{enumerate}
    \item Radial weights $\nu(z)=(1-|z|^2)^\alpha$, \ $\alpha>-1$.
    \item B\'ekoll\'e--Bonami ${\bf B}_p$ weights for $1<p<\infty$ (see, e.g., \cite{DMO2016}).
    \item Weights $\eta$ that satisfy both bounded hyperbolic oscillation, i.e., there exists $C_\eta>0$ such that for every arc $I\subseteq\mathbb T$,
    \[
    C_\eta^{-1}\eta(\xi)\le \eta(z)\le C_\eta\,\eta(\xi), \qquad z,\xi\in Q_I^{\mathrm{up}},
    \]
    and the Fuji--Wilson property
    \[
    \sup\left\{\frac{\int_{Q_I} M(\eta\,\one_{Q_I})}{\int_{Q_I}\eta}: I\subseteq\mathbb T\right\}<+\infty
    \]
    (see \cite[Theorem~1.7]{APM2019}).
\end{enumerate}
To this end, we refer the interested reader to the recent excellent papers \cite{APM2019,MP2025} and the references therein for a more systematic study of arbitrary B\'ekoll\'e--Bonami weights.
\end{rem}

We have the following result.

\begin{thm} \label{20260114thm01a}
Let $1\le q<p\le +\infty$ and $t>0$. Let $\mathcal D$ be a dyadic system on $\mathbb T$. 
Let $\mu$ and $\omega$ be two weights on $\mathbb D$, and set $\sigma:=\omega^{-1/(p-1)}$. 
Assume that $\mu, \sigma \in \mathbf B_\infty(\calD)$.  Then the following statements are equivalent.

\begin{enumerate}
    \item $\calM_t^{\calD}$ extends a bounded operator from $L^p(\omega, \D)$ to $L^q(\mu, \D)$;
    \item Define 
    $$
\phi(z):=\sum_{I \in \calD} \beta_I \one_{Q_I^{\textrm{up}}}(z)
$$
where
$$
\beta_I:= \frac{1}{|Q_I|^{(t-1)q}} \cdot \frac{\mu(Q_I)}{|Q_I|} \cdot \left(\frac{\sigma(Q_I)}{|Q_I|} \right)^{\frac{q}{p'}}.
    $$
    Then $\phi \in L^{\frac{p}{p-q}}(\D)$. 
\end{enumerate}
\end{thm}

\begin{proof} 
\textbf{$(2)\Longrightarrow (1)$.} \; Assume (2). For any $\alpha>0$, as usual, let $E_\alpha:=\{z \in \D: \calM_t^{\calD}f(z)>\alpha\}$. Again, write $E_\alpha$ into disjoint union of maximal Carleson boxes $\{Q_{\alpha, i}\}_{i \ge 1}$ with
$$
\frac{1}{|Q_{\alpha, i}|^t} \int_{Q_{\alpha, i}} |f(z)|dA(z)>\alpha, \qquad \textrm{for} \quad i \ge 1. 
$$
Then, using the assumption that $\sigma \in {\bf B}_\infty(\calD)$, 
\begin{align} \label{20260117eq02}
\alpha^q \mu(E_\alpha)
&=\alpha^q \sum_{i=1}^\infty \mu(Q_{\alpha, i}) \lesssim \sum_{i=1}^\infty \frac{\mu(Q_{\alpha, i})}{|Q_{\alpha, i}|^{tq}} \left(\int_{Q_{\alpha, i}} |f(z)|dA(z) \right)^q \nonumber  \\
&=\sum_{i=1}^\infty \frac{\mu(Q_{\alpha, i})}{|Q_{\alpha, i}|^{tq}} \left[ \sigma(Q_{\alpha, i}) \cdot \frac{1}{\sigma(Q_{\alpha, i})} \int_{Q_{\alpha, i}} |f(z)| \sigma^{-1}(z)\sigma(z) dA(z) \right]^q \nonumber  \\
& \lesssim \sum_{i=1}^\infty \frac{\mu(Q_{\alpha, i})}{|Q_{\alpha, i}|^{tq}} \left[ \sigma(Q^{\textrm{up}}_{\alpha, i}) \cdot \frac{1}{\sigma(Q_{\alpha, i})} \int_{Q_{\alpha, i}} |f(z)| \sigma^{-1}(z)\sigma(z) dA(z) \right]^q \nonumber \\
& \lesssim \sum_{i=1}^\infty \frac{\mu(Q_{\alpha, i})}{|Q_{\alpha, i}|^{tq}} \left[ \int_{Q_{\alpha, i}^{\textrm{up}}} \calM^{\calD, \sigma} \left(|f|\sigma^{-1} \right)(z) \sigma(z) dA(z) \right]^q,
\end{align}
where $\calM^{\calD, \sigma}$ is the weighted dyadic maximal operator given by 
$$
\calM^{\calD, \sigma}f(z):=\sup_{z \in Q_I, \; I \in \calD} \frac{1}{\sigma(Q_I)} \int_{Q_I} |f(z)|\sigma(z)dA(z). 
$$
Next, we estimate $\left\|\calM_t^{\calD} f \right\|_{L^q(\mu, \D)}$. First, write 
\begin{equation}  \label{20260117eq01}
\int_{\D} \left|\calM_t^{\calD} f(z) \right|^q \mu(z)dA(z) 
 \lesssim \sum_{\ell \in \Z} 4^{t\ell} \mu \left(E_{4^{t\ell}} \backslash E_{4^{t(\ell+1)}} \right).
\end{equation} 
Again, denote 
$$
\left\{ \widetilde{Q}_{4^{t\ell}, i} \right\}_{i \ge 1}:= \left\{Q_{4^{t\ell}, i} \right\}_{i \ge 1} \backslash \left\{Q_{4^{t(\ell+1)}, i} \right\}_{i \ge 1}.
$$
Therefore, using \eqref{20260117eq02} with $E_\alpha$ replaced by $E_{4^{t\ell}} \backslash E_{4^{t(\ell+1)}}$, and $\left\{Q_{\alpha, i} \right\}_{i \ge 1} $ by $\left\{ \widetilde{Q}_{4^{t\ell}, i} \right\}_{i \ge 1}$, respectively, we derive that 
\begin{align} \label{20260114eq01}
& \textrm{RHS of \eqref{20260117eq01}}  \nonumber \\
& \lesssim \sum_{\ell \in \Z}  \sum_{i=1}^\infty \frac{\mu(\widetilde{Q}_{4^{t\ell}, i})}{|\widetilde{Q}_{4^{t\ell}, i}|^{tq}} \left( \int_{\widetilde{Q}_{4^{t\ell}, i}^{\textrm{up}}} \calM^{\calD, \sigma} \left(|f|\sigma^{-1} \right)(z) \cdot  \sigma^{\frac{1}{p}}(z)\sigma^{\frac{1}{p'}}(z) dA(z) \right)^q \nonumber \\
& \lesssim \sum_{\ell \in \Z} \sum_{i=1}^\infty \frac{\mu(\widetilde{Q}_{4^{t\ell}, i})}{|\widetilde{Q}_{4^{t\ell}, i}|^{tq}} \left( \int_{\widetilde{Q}_{4^{t\ell}, i}^{\textrm{up}}} \left|\calM^{\calD, \sigma}(|f|\sigma^{-1})(z) \right|^p \sigma(z) dA(z) \right)^{\frac{q}{p}} \cdot \left( \int_{\widetilde{Q}_{4^{t\ell}, i}^{\textrm{up}}} \sigma(z) dA(z) \right)^{\frac{q}{p'}}  \nonumber\\
& =  \sum_{\ell \in \Z} \sum_{i=1}^\infty \frac{\mu(\widetilde{Q}_{4^{t\ell}, i}) \sigma^{\frac{q}{p'}}(\widetilde{Q}_{4^{t\ell}, i}^{\textrm{up}})}{|\widetilde{Q}_{4^{t\ell}, i}|^{tq}} \cdot  \left( \int_{\widetilde{Q}_{4^{t\ell}, i}^{\textrm{up}}} \left|\calM^{\calD, \sigma}(|f|\sigma^{-1})(z) \right|^p \sigma(z) dA(z) \right)^{\frac{q}{p}} \nonumber \\
& \le \left( \sum_{\ell \in \Z} \sum_{i=1}^\infty \int_{\widetilde{Q}_{4^{t\ell}, i}^{\textrm{up}}} \left|\calM^{\calD, \sigma}(|f|\sigma^{-1})(z) \right|^p \sigma(z) dA(z) \right)^{\frac{q}{p}} \cdot \left( \sum_{\ell \in \Z} \sum_{i=1}^\infty  \left( \frac{\mu(\widetilde{Q}_{4^{t\ell}, i}) \sigma^{\frac{q}{p'}}(\widetilde{Q}_{4^{t\ell}, i}^{\textrm{up}})}{|\widetilde{Q}_{4^{t\ell}, i}|^{tq}} \right)^{\frac{p}{p-q}} \right)^{\frac{p-q}{p}}.
\end{align}
Now for the first double summation in \eqref{20260114eq01}, we have 
\begin{align} \label{20260114eq02}
    \sum_{\ell \in \Z} \sum_{i=1}^\infty \int_{\widetilde{Q}_{4^{t\ell}, i}^{\textrm{up}}} \left|\calM^{\calD, \sigma}(|f|\sigma^{-1})(z) \right|^p \sigma(z) dA(z)
    & \lesssim \int_{\D}  \left|\calM^{\calD, \sigma}(|f|\sigma^{-1})(z) \right|^p \sigma(z) dA(z) \nonumber \\
    & \lesssim \int_{\D} \left| f(z) \sigma^{-1}(z) \right|^p \sigma(z) dA(z) \nonumber \\
    & \lesssim \int_{\D} |f(z)|^p \omega(z) dA(z),
\end{align}
where in the second estimate above, we have used the boundedness of $\calM^{\calD, \sigma}$ acting on $L^p(\sigma, \D)$ for any $1<p \le +\infty$. 

While for the second double summation in \eqref{20260114eq01}, we observe that 
\begin{align} \label{20260114eq03}
     \sum_{\ell \in \Z} \sum_{i=1}^\infty  \left( \frac{\mu(\widetilde{Q}_{4^{t\ell}, i}) \sigma^{\frac{q}{p'}}(\widetilde{Q}_{4^{t\ell}, i}^{\textrm{up}})}{|\widetilde{Q}_{4^{t\ell}, i}|^{tq}} \right)^{\frac{p}{p-q}}
     &\simeq \sum_{\ell \in \Z} \sum_{i=1}^\infty |\widetilde{Q}_{4^{t\ell}, i}^{\textrm{up}}| \left( \frac{\mu(\widetilde{Q}_{4^{t\ell}, i}) \sigma^{\frac{q}{p'}}(\widetilde{Q}_{4^{t\ell}, i}^{\textrm{up}})}{|\widetilde{Q}_{4^{t\ell}, i}|^{tq}  |\widetilde{Q}_{4^{t\ell}, i}|^{\frac{p-q}{p}}}\right)^{\frac{p}{p-q}} \nonumber \\
     & \lesssim \sum_{\ell \in \Z} \sum_{i=1}^\infty \int_{\widetilde{Q}_{4^{t\ell}, i}^{\textrm{up}}} \phi^{\frac{p}{p-q}}(z)dA(z) \le  \int_{\D} \phi^{\frac{p}{p-q}}(z)dA(z). 
\end{align}
Plugging \eqref{20260114eq02} and \eqref{20260114eq03} back to \eqref{20260114eq01}, we derive that 
$$
\int_{\D} |\calM_t^{\calD}f(z)|^q \mu(z)dA(z) \lesssim \left\| f \right\|_{L^p(\omega, \D)}^q \left\| \phi \right\|_{L^{\frac{p}{p-q}}(\D)}, 
$$
which gives $(2) \Longrightarrow (1)$. 

\medskip 

\textbf{$(1) \Longrightarrow (2)$.} \; Let $\{a_I\}_{I \in \calD}$ be any sequence of non-negative numbers, and 
$$
F(z):=\sum_{I \in \calD} a_I \sigma(z) \one_{Q_I^{\textrm{up}}} (z). 
$$
First, we note that
\begin{equation} \label{20260118eq01}
\int_\D |F(z)|^pw(z)dA(z)= \sum_{I \in \calD} a_I^p \int_{Q_I^{\textrm{up}}} \sigma^p(z) w(z)dA(z)=\sum_{I \in \calD} a_I^p \sigma(Q_I^{\textrm{up}}) \lesssim \sum_{I \in \calD} a_I^p \sigma(Q_I). 
\end{equation} 
On the other hand, for any $z \in Q_I^{\textrm{up}}$, one has
$$
\calM_t^{\calD} F(z) \ge \frac{1}{|Q_I|^t} \int_{Q_I} |F(z)|dA(z) \ge  \frac{1}{|Q_I|^t} \int_{Q^{\textrm{up}}_I} |F(z)|dA(z)=\frac{a_I \sigma(Q_I^{\textrm{up}})}{|Q_I|^t} \gtrsim \frac{a_I \sigma(Q_I)}{|Q_I|^t},
$$
where in the last estimate above, we have used the assumption that $\sigma \in {\bf B}_\infty(\calD)$. Therefore, 
\begin{align*}
\int_{\D} \left| \calM_t^{\calD} F(z) \right|^q \mu(z)dA(z) 
&  \gtrsim \sum_{I \in \calD} \int_{Q_I^{\textrm{up}}} \left(\frac{a_I \sigma(Q_I)}{|Q_I|^t} \right)^q \mu(z)dA(z) \\
& \gtrsim \sum_{I \in \calD} \frac{a_I^q \sigma^q(Q_I)}{|Q_I|^{tq}} \mu (Q_I^{\textrm{up}}) \\
& \gtrsim \sum_{I \in \calD} \frac{a_I^q \sigma^q(Q_I)}{|Q_I|^{tq}} \mu (Q_I) \\
& =\sum_{I \in \calD} \frac{\mu(Q_I) \sigma^{\frac{q}{p'}}(Q_I)}{|Q_I|^{tq}} \cdot \left(a_I^p \sigma(Q_I) \right)^{\frac{q}{p}} \\
&=\sum_{I \in \calD} \left(\beta_I |Q_I|^{1-\frac{q}{p}} \right) \cdot \left(a_I^p \sigma(Q_I) \right)^{\frac{q}{p}},
\end{align*}
where in the third estimate above, we use the fact that $\mu \in {\bf B}_\infty(\calD)$. Since $\calM_t^{\calD}: L^p(\omega, \D) \to L^q(\mu, \D)$ is bounded, using the above estimate together with \eqref{20260118eq01}, we derive that
\begin{equation} \label{20260118eq02}
\sum_{I \in \calD} \left(\beta_I |Q_I|^{1-\frac{q}{p}} \right) \cdot \left(a_I^p \sigma(Q_I) \right)^{\frac{q}{p}} \lesssim \left( \sum_{I \in \calD} a_I^p \sigma(Q_I) \right)^{\frac{q}{p}}, 
\end{equation} 
where the implicit constant in the above estimate is independent of the choice of $\{a_I\}_{I \in \calD}$. Since $p>q$ and the choice of $\{a_I\}_{I\in\mathcal D}$ is arbitrary, duality yields that  
$$
1 \gtrsim \sum_{I \in \calD} \left(\beta_I |Q_I|^{1-\frac{q}{p}} \right)^{\frac{p}{p-q}} \simeq \sum_{I \in \calD} |Q_I^{\textrm{up}}| \beta_I^{\frac{p}{p-q}}=\left\|\phi \right\|_{L^{\frac{p}{p-q}}(\D)}^{\frac{p}{p-q}}. 
$$
The proof is complete. 
\end{proof}

\medskip

\section{Critical line estimates for the hypersingular Bergman projection} \label{20251219sec04}

Our next goal is to establish critical line estimates for the hypersingular Bergman projection: for $1<t<3/2$, 
$$
K_{2t} f(z):=\int_{\D} \frac{f(w)}{(1-z\overline{w})^{2t}} dA(w),
$$
which can be regarded as a singular integral counterpart of the hypersingular maximal operator $\calM_t^{\calD}$. We first recall that the off-critical line $L^p$ theory was studied in \cite[Theorem 3]{CFWY2017}. In particular, they showed that for $K_{2t}$ is bounded from $L^p(\D)$ to $L^q(\D)$ if and only $1/q-1/p>2t-2$ for $1 \le p, q \le \infty$, which is exactly the off-critical line regime for $\calM_t^{\calD}$. Therefore, it is natural to consider the behavior of $K_{2t}$ on the critical line 
\begin{equation} \label{20260201eq01}
\left\{\left(\frac{1}{p}, \frac{1}{q} \right)  \in [0, 1]^2: \frac{1}{q}-\frac{1}{p}=2t-2 \right\}. 
\end{equation} 
We first have the following observation. 

\begin{lem} \label{20251225lem01}
For any $1<t<3/2$, $K_{2t}: L^\infty(\D) \to L^{\frac{1}{2t-2}, \infty}(\D)$ is bounded. 
\end{lem}

\begin{proof}
The proof of this lemma is straightforward. Let $f \in L^\infty(\D)$ with $\left\|f \right\|_{L^\infty(\D)}=1$. Observe now that
\begin{align*}
\left|K_{2t}f(z) \right|=\left| \int_{\D} \frac{f(w)}{(1-z\overline{w})^{2t}} dA(w) \right|  \lesssim \int_{\D} \frac{1}{|1-z\overline{w}|^{2t}} dA(w) \simeq \frac{1}{(1-|z|^2)^{2(t-1)}}, 
\end{align*}
where in the last estimate, we used the standard integral estimate \cite[Theorem 1.12]{Zhu2005}. Therefore, for any $\alpha>0$, as in Lemma \ref{20251219lem02}, 
$$
\left|\left\{z \in \D: \left|K_{2t}f(z) \right|>\alpha \right\} \right| \lesssim \alpha^{-\frac{1}{2(t-1)}}, 
$$
which yields the desired weak-type bounds. 
\end{proof}

As a consequence, we show that $K_{2t}$ satisfies a weak-type bound only in the region on and above the critical line~\eqref{20260201eq01}.

\begin{prop} \label{20260131prop02}
For any $1<t<3/2$ and $1 \le p, q \le +\infty$ with $\frac{1}{q}-\frac{1}{p}<2t-2$, $K_{2t}$ maps $L^p(\D)$ unboundedly to $L^{q, \infty}(\D)$.
\end{prop}

\begin{proof}
Without loss of generality, we may assume that \(q<+\infty\). The unboundedness in the case \(q=+\infty\) follows immediately from \cite[Theorem~3]{CFWY2017}.

Assume that \(K_{2t}\) maps \(L^p(\D)\) boundedly into \(L^{q,\infty}(\D)\). First, we consider the case $q>1$. Fix \(\varepsilon>0\) sufficiently small so that $q-\varepsilon>1$ and $1/(q-\varepsilon)-1/p<2t-2$. Then using the fact that $L^{q, \infty}(\D) \subseteq L^{q-\eps}(\D)$ (see, e.g., \cite[Exercise~1.1.11]{Grafakos2014}), we have
$$
\|K_{2t}f\|_{L^{q-\varepsilon}(\D)}
\lesssim
\|K_{2t}f\|_{L^{q,\infty}(\D)}
\lesssim
\|f\|_{L^p(\D)}.
$$
This contradicts \cite[Theorem~3]{CFWY2017}.

Finally, if $q=1$, then by applying the off-diagonal Marcinkiewicz interpolation to the bounds
$K_{2t}: L^p(\D)\to L^{1,\infty}(\D)$ and $K_{2t}: L^\infty(\D)\to L^{\frac{1}{2t-2},\infty}(\D)$
(see Lemma~\ref{20251225lem01}), one can choose exponents $1<\widetilde p,\widetilde q<+\infty$ such that
$K_{2t}$ maps $L^{\widetilde p}(\D)$ boundedly into $L^{\widetilde q,\infty}(\D)$, where
$(1/\widetilde p,\,1/\widetilde q)$ lies on the line segment joining $(0,\,2t-2)$ and $(1/p,\,1)$.
The same argument as in the case $q>1$ applies, and the proof is complete.

\end{proof}

Our next goal is to establish weak-type estimates at the other endpoint
$\bigl(\tfrac{1}{3-2t},1\bigr)$.
We have the following result.

\begin{prop} \label{20251226thm01}
For any $1<t<3/2$, $K_{2t}: L^{\frac{1}{3-2t}}(\D) \to L^{1, \infty}(\D)$ is bounded. 
\end{prop}

\begin{proof}
Let $f \in L^{\frac{1}{3-2t}}(\D)$. Write
$$
K_{2t} f(z)=W(z)B_{2t} f(z), 
$$
where $W(z):=(1-|z|^2)^{2-2t}$ and $B_{2t}$ is the Forelli--Rudin type operator given by 
$$
B_{2t} f(z):=(1-|z|^2)^{2t-2} \int_{\D} \frac{f(w)}{(1-z \overline{w})^{2t}} dA(w). 
$$
We make the following claims. 
\begin{enumerate}
    \item [(1)] $B_{2t}: L^{\frac{1}{3-2t}}(\D) \to L^{\frac{1}{3-2t}}(\D)$ is bounded; 
    \item [(2)] $W \in L^{\frac{1}{2t-2}, \infty}(\D)$. 
\end{enumerate}
Assuming (1) and (2), and using Lorentz--H\"older inequality, we have 
\begin{align*}
\left\|K_{2t} f \right\|_{L^{1, \infty}(\D)}
&=\left\| W \cdot B_{2t} f\right\|_{L^{1, \infty}(\D)} \lesssim \left\|W \right\|_{L^{\frac{1}{2t-2}, \infty}(\D)} \left\|B_{2t} f \right\|_{L^{\frac{1}{3-2t}}(\D)} \\
& \lesssim \left\|W \right\|_{L^{\frac{1}{2t-2}, \infty}(\D)} \left\|f \right\|_{L^{\frac{1}{3-2t}}(\D)} \lesssim\left\|f \right\|_{L^{\frac{1}{3-2t}}(\D)},  
\end{align*}
which gives the desired result. 

Therefore, it remains to verify claims (1) and (2). Claim (1) follows from \cite[Theorem~3]{Zhao2015} by choosing $
a=2t-2, b=0, c=2t, n=1,  \alpha=\beta=0$, and $p=q=1/(3-2t)$ there. Claim (2) follows from a direct computation.
\end{proof}

\begin{rem}
The proof of the above proposition illustrates what we will call the \emph{Forelli--Rudin method}:
one reduces the desired weak-type estimate to a weight multiplied by a less singular Forelli--Rudin type operator. While the proof of Proposition \ref{20251226thm01} is carried out in a complex-analytic way, we will see later that the same idea can be adapted to obtain endpoint estimates for certain hypersingular averaging operators (see, Theorem \ref{20251227thm01}).

\end{rem}

Combining Lemma \ref{20251225lem01} and Proposition \ref{20251226thm01} with off-diagonal Marcinkiewicz interpolation theorem, we have the following. 

\begin{cor} \label{20251228cor01}
For any $(p, q)$ belonging to the critical line regime, that is,
$$
\left\{\left(\frac{1}{p}, \frac{1}{q} \right)  \in  (0, 1)^2: \frac{1}{q}-\frac{1}{p}=2t-2 \right\},
$$
then for every $0<r<\infty$, the operator $K_{2t}$ extends to a bounded map 
$$
K_{2t}: L^{p,r}(\D)\to L^{q,r}(\D),
$$
In particular, $K_{2t}: L^p(\D) \to L^{q, \infty}(\D)$ is bounded. 
\end{cor}

\medskip 

\section{Hypersingular sparse operators} \label{20251219sec05}

A natural question arising from the study of $\calM_t^{\calD}$ and $K_{2t}$ is whether there exists a unified harmonic-analytic framework that treats these two complex-analytic models simultaneously, and that extends to a broader class of hypersingular operators.
It turns out that a more general principle underlies this phenomenon. 

\medskip

Let us now turn to some details. We recall some definitions first. 

\begin{defn} \label{20251227defn01a}
Let $\calD$ be a dyadic system in $\R^n$ and $\calS \subset \calD$ be a collection of dyadic cubes. For $0<\eta<1$, we say $\calS$ is $\eta$-sparse, if there exists a collection of measurable sets $\{E(Q)\}_{Q \in \calS}$, such that
\begin{enumerate}
    \item $E(Q) \subseteq Q$ for all $Q \in \calS$;
    \item $|E(Q)| \ge \eta |Q|$ for all $Q \in \calS$;
    \item $E(Q) \cap Q'=\emptyset$ for any $Q' \subsetneq Q$, $Q' \in \calS$.
\end{enumerate}
\end{defn}

\begin{rem}
We remark that the above notion of sparse families was introduced by Lerner, Lorist, and Ombrosi \cite{LLO2022} in their work on operator-free sparse domination. Compared with the usual definition used in the literature, this formulation is slightly more restrictive, most through the third condition above, which we refer to as the \emph{contracting property}. A simple observation shows that this contracting property forces the sets $\{E(Q)\}_{Q\in\mathcal S}$ to be pairwise disjoint, thereby recovering the usual disjointness requirement in the standard definition of a sparse family.

Such sparse collections arise naturally in many applications of sparse domination, for instance: (1) collections of Carleson boxes; and (2) sparse collections produced by stopping-time constructions, such as those associated with maximal operators and Calder\'on--Zygmund operators. We refer the reader to \cite{LLO2022} for further discussion and applications of sparse families with the contracting property.

\end{rem}

We begin with the following model case. Let $t>1$, and for simplicity, let $\mathcal S$ be an $\eta$-sparse family of cubes in $\mathbb R^n$ such that $Q \subseteq Q_0=[0, 1]^n$ for all $Q \in \calS$. We define the \emph{(pointwise) hypersingular sparse operator associated with $\mathcal S$} by
\[
\mathbb A_{\mathcal S}^t f(x)
:=\sum_{Q\in\mathcal S}\frac{\one_Q(x)}{|Q|^{t-1}}\langle |f|\rangle_Q
=\sum_{Q\in\mathcal S}\frac{\one_Q(x)}{|Q|^{t}}\int_Q |f(y)|\,dy.
\]

\begin{rem} \label{20251229rem02}

The normalization $Q_0=[0,1]^n$ is made only for convenience. In general, it suffices to assume that there exists a fixed dyadic cube $Q_0$ such that $Q\subseteq Q_0$ for all $Q\in\calS$; by a translation and dilation one may then reduce to the above normalized situation. 

This global containment assumption is natural in the present hypersingular regime. Indeed, as noted in Remark~\ref{20251227rem01}(2), when $t>1$ the operator $\mathbb A_{\calS}^t$ may not even be well-defined on nonzero constant functions whenever it is not localized. On the other hand, from the viewpoint of dyadic harmonic analysis, the assumption is mild: most operators of interest (such as Calder\'on--Zygmund operators and Hilbert transforms along monomial curves) are local, or can be decomposed into a sum of localized pieces, and this yields the existence of such a global cube $Q_0$ (for each localized piece).

\end{rem}

A natural question is the following.

\begin{ques}
Let $\eta\in(0,1)$, $t>1$, and $\mathcal S$ be an $\eta$-sparse family in $\mathbb R^n$ as above. 
For which pairs $(p,q)$ with $1\le p,q\le\infty$ does $\mathbb A_{\mathcal S}^t$ extend to a bounded operator
\[
\mathbb A_{\mathcal S}^t: L^{p}(\mathbb R^n)\to L^{q}(\mathbb R^n)
\qquad\text{or}\qquad
\mathbb A_{\mathcal S}^t: L^{p}(\mathbb R^n)\to L^{q,\infty}(\mathbb R^n)?
\]

\end{ques}

Our goal in the remainder of this section is to address this question.

\subsection{Graded family}  \label{20251229subsec01}

 It turns out that, in addition to the sparseness of $\calS$, there is another fundamental structural parameter that influences the behavior of the hypersingular sparse operator $\mathbb A_{\calS}^t$. 
 
The \emph{key} observation is already contained in Example~\ref{20251227obs01}, which shows that one must control how the sizes of cubes in $\calS$ change from one ``layer'' to the next. This motivates us to introduce the notion of a \emph{graded family} of dyadic cubes.

\vspace{0.1cm}

We now turn to some details. Let $\calG \subseteq \calD$ be a collection of dyadic cubes. Again, we may assume $Q \subseteq [0, 1]^n$ for all $Q \in \calG$. 

\vspace{0.1cm} 

\noindent $\bullet$ First, let $\calG^{(0)}$ denote the collection of all maximal dyadic cubes in $\calG$. For simplicity, we may assume that $\calG^{(0)}=\{[0, 1]^n \}$. Otherwise, we decompose $\calG$ into finitely many such collections and treat each one separately (by translation and dilation). To this end, we define $\mathfrak G_0:=\ell(Q_0)$, where $\ell(Q)$ denotes the sidelength of a dyadic cube $Q$.

\vspace{0.1cm}

\noindent $\bullet$ Next, let $\calG^{(1)}$ be the subcollection of all maximal dyadic cubes in $\calG \setminus \calG^{(0)}$, and define $\mathfrak G_1:=\inf_{Q \in \calG^{(1)}}\ell(Q)$. Iterating this procedure, we obtain a decomposition of $\calG$ into layers $\{\calG^{(j)}\}_{j\ge 0}$ together with the associated scales $\{\mathfrak G_j\}_{j\ge 0}$. Observe that for each $j \ge 1$, the dyadic cubes in $\calG^{(j)}$ are mutually disjoint. 

\begin{defn}\label{20251227defn01}
Let $\calG \subseteq \calD$ be a collection of dyadic cubes in $\R^n$ such that $\calG^{(0)}=\{[0,1]^n\}$, and write $\calG=\bigcup_{j \ge 0}\calG^{(j)}$ as above. We say that $\calG$ is \emph{graded} if
\[
K_{\calG}:=\sup_{j \ge 0}\left(\log_2 \frac{\mathfrak G_j}{\mathfrak G_{j+1}}\right)<\infty.
\]
We call $K_{\calG}$ the \emph{degree} of $\calG$, and we refer to $\calG^{(j)}$ as the \emph{$j$-th layer} of $\calG$.
\end{defn}

\begin{rem} \label{20251229rem01}
\begin{enumerate}
    \item Here, we may assume that $\calG^{(1)} \neq \emptyset$. Otherwise, $\mathbb A_{\calG}^t f(x)=\one_{[0, 1]^n}(x)\int_{[0, 1]^n} f$, which is a rank-one operator and, in particular, maps $L^1(\R^n)$ boundedly into $L^\infty(\R^n)$.

    \vspace{0.1cm}

    \item Here are some examples of graded family. Consider the collection of all dyadic cubes contained in $[0,1]^n$ together; then such a collection of dyadic cubes is graded with degree $1$. Another example is the collection\footnote{In that setting, the role of $Q_0$ is replaced by $\D$, and the sidelength $\ell(Q)$ in Definition~\ref{20251227defn01} is replaced by the length of the boundary arc associated with the Carleson box.} of dyadic Carleson boxes that appeared in our earlier analysis of $\calM_t^{\calD}$ and $K_{2t}$. This family is also graded, again with degree $1$. Finally, we observe that $K_{\calG} \ge 1$. 

    \vspace{0.1cm}

    \item It is clear that a graded family need \emph{not} be sparse, and vice versa.

    \vspace{0.1cm}

    \item It is not correct to replace the $\sup_{j \ge 0}$ in the definition of $K_{\calG}$ by the quantity
\[
K_{\calG}':=\limsup_{k\to\infty}\left(\log_2 \frac{\mathfrak G_j}{\mathfrak G_{j+1}}\right).
\]
Indeed, $K_{\calG}'$ only controls the ratios $\mathfrak G_j/\mathfrak G_{j+1}$ for sufficiently large $j$ and, in particular, imposes no restriction on the initial scales (for instance, on $\mathfrak G_0/\mathfrak G_1$). Consequently, one loses uniform control on the gaps between the first few layers. Using the same idea as in Observation~\ref{20251227obs01}, one can construct a sequence of sparse families with the same sparseness and the same value of $K_{\calG}'$, but for which the associated hypersingular sparse operators still exhibit the ``blow-up" phenomenon described in Example~\ref{20251227obs01}.
\end{enumerate}
\end{rem}

\subsection{$L^p$ theory for hypersingular sparse operator $\mathbb A_{\calS}^t$ induced by graded family}

We have the following result, whose proof relies on a dyadic version of the Forelli--Rudin method.

\begin{thm} \label{20251227thm01}
Let $\eta \in (0, 1)$, $\calS \subseteq \calD$ be a graded (contracting) $\eta$-sparse family in $[0, 1]^n$ with degree $K_{\calS}$ as in Definition \ref{20251227defn01}, $1<t<1-\frac{\log_2(1-\eta)}{nK_{\calS}}$, and $\mathbb A_{\calS}^t$ be the associated hypersingular sparse operator.  Then the following statements hold.
\begin{enumerate}
    \item (Off-critical line estimate) $\mathbb A_{\calS}^t$ extends to a bounded operator from $L^p(\R^n)$ to $L^q(\R^n)$ when $(p, q)$ belongs to the off-critical line regime associated to $\mathbb A_{\calS}^t$ given by 
    \begin{equation} \label{20251228eq01}
    \left\{\left(\frac{1}{p}, \frac{1}{q} \right) \in [0, 1]^2: \frac{1}{q}-\frac{1}{p}>\frac{nK_{\calS}(t-1)}{-\log_2(1-\eta)} \right\}. 
    \end{equation} 
    \item (Critical line estimate) $\mathbb A_{\calS}^t$ extends to a bounded operator from $L^{p}(\R^n)$ to $L^{q, \infty}(\R^n)$ when $(p, q)$ belongs to 
        \begin{equation} \label{20251228eq30}
        \left\{\left(\frac{1}{p}, \frac{1}{q} \right) \in [0, 1]^2: \frac{1}{q}-\frac{1}{p}=\frac{nK_{\calS}(t-1)}{-\log_2(1-\eta)}  \right\}. 
        \end{equation} 
\end{enumerate}
\end{thm}

Before we prove Theorem \ref{20251227thm01}, we make some remarks.

\begin{rem} \label{20251230rem01}
Theorem~\ref{20251227thm01} strengthens our earlier results on the boundedness behavior of $K_{2t}$
(see, Section~\ref{20251219sec04}). Moreover, the proof of \textbf{Step~II} in
Theorem~\ref{20251227thm01} yields an alternative proof of \cite[Theorem~3]{CFWY2017}, and the proof of \textbf{Step~III} in in
Theorem~\ref{20251227thm01} offers a different proof of Proposition \ref{20251226thm01}. Moreover, Theorem~\ref{20251227thm01} is sharp, since it recovers the \(L^p\) theory for \(\calM_t^{\calD}\) and \(K_{2t}\) as special cases.

\end{rem} 

\begin{rem} \label{20251230rem01}
We make a further remark on the sparseness parameter. In our set-up, the sparseness
$\eta\in(0,1)$ should always be understood \emph{with respect to the underlying grid}.
For the dyadic grids (base $2$), it is convenient to encode $\eta$ by
\[
\kappa:=-\log_2(1-\eta),
\]
where the base $2$ reflects the dyadic structure.

This normalization is stable under changing the base of the grid. For example,
consider the Carleson boxes associated with a triadic system on $\T$
(i.e.\ in \eqref{20251230eq01} we replace dyadic arcs of length $2^{-j}$ by triadic arcs
of length $3^{-j}$). Let $\mathcal I$ be a triadic arc, then
the corresponding ``upper'' region is given by
\[
Q_{\mathcal I}^{\mathrm{up}}
:=\left\{z\in\D:\ \frac{z}{|z|}\in \mathcal I,\ 1-|\mathcal I|\le |z|<1-\frac{|\mathcal I|}{3}\right\},
\]
and hence the collection of all Carleson boxes associated with a triadic system is sparse
with sparseness $\eta=2/3$. Therefore, if we normalize the sparseness using the base of
the grid, then
\[
-\log_3\Bigl(1-\frac23\Bigr)=1,
\]
which coincides with the dyadic normalization $-\log_2(1-1/2)=1$.

\end{rem}

\begin{proof}[Proof of Theorem \ref{20251227thm01}]
For each $j \ge 0$, let $\calS^{(j)}$ denote the $j$-th layer of $\calS$ as in Definition \ref{20251227defn01}, and without loss of generality, we may assume $\calS^{(0)}=\{[0, 1]^n\}$. 

Observe that, by the sparseness assumption and contracting property of $\calS$, for any $j \ge 0$, one has
\begin{equation} \label{20251226eq30}
|\mathfrak D_j|=\sum_{Q \in \calS^{(j)}} |Q|  \lesssim (1-\eta)^j, 
\end{equation} 
where we write $\mathfrak D_j:=\bigcup_{Q \in \calS^{(j)}} Q$. We divide the proof into several steps.

\vspace{0.1cm}

\noindent {\bf Step I: Weak-type bounds at the point $\left(1/p, 1/q \right)=\left(0, \frac{nK_{\calS}(t-1)}{-\log_2(1-\eta)} \right)$.} Our goal is to show the boundedness of 
\begin{equation} \label{20251227eq80}
\mathbb A_{\calS}^t: L^\infty(\R^n) \to L^{\frac{-\log_2(1-\eta)}{nK_{\calS}(t-1)}, \infty}(\R^n). 
\end{equation} 
Let $\alpha>0$ and $f \in L^\infty(\R^n)$ with $\|f \|_{L^\infty(\R^n)}=1$. Then, we have to estimate the size of the level set $E:=\left|\left\{x \in [0, 1]^n: \left|\mathbb A_{\calS}^t f (x)\right|>\alpha \right\} \right|$. Without loss of the generality, we may assume $\alpha$ is sufficiently large. We have the following observation.

\medskip 

\noindent (1). First, consider $\calS^{(0)}=\{[0, 1]^n\}$. Observe that
$$
\mathbb A_{\calS}^t f(z) \le \frac{1}{\left|[0, 1]^n \right|^t} \int_{[0, 1]^n} |f(x)|dx \le 1, \qquad z \in [0, 1]^n \backslash \mathfrak D_1. 
$$
By the sparseness and contracting property of $\calS$, $\mathbb A_{\calS}^t f(z)$ can only take larger values on $\mathfrak D_1$, whose size is at most $1-\eta$. 

\medskip 

\noindent (2). Next, we consider the next layer $\calS^{(1)}$. Using the sparseness and contracting property of $\calS$ again, we find that
$$
\mathbb A_{\calS}^t f(z) \le 1+\frac{1}{|Q|^t} \int_{Q} |f(x)| dx \le 1+|Q|^{1-t} \le 1+2^{nK_{\calS}(t-1)}, \quad z \in [0, 1]^n \backslash \mathfrak D_2
$$
and $\mathbb A_{\calS}^t f(z)$ can only take larger values on $\mathfrak D_2$, 
whose size is at most $(1-\eta)^2$. 

\medskip 

Iterating the above procedure, we see that for any $J \ge 0$, if 
$$
\mathbb A_{\calS}^t f(z) \le \sum_{\ell=0}^J 2^{n\ell K_{\calS}(t-1)}=C_1 2^{nJK_{\calS}(t-1)}-C_2, 
$$
where $C_1, C_2>0$ are some absolute constants\footnote{Here, we can take $C_1=\frac{2^{nK_{\calS}(t-1)}}{2^{nK_{\calS}(t-1)}-1}$ and $C_2=\frac{1}{2^{nK_{\calS}(t-1)}-1}$. } that only depend on $n, K_{\calS}$ and $t$, then $z \in [0, 1]^n \backslash \mathfrak D_{J+1} $, whose size is at least $1-(1-\eta)^{J+1}$. 

Therefore, for any $\alpha>0$ sufficiently large, if $\mathbb A_{\calS}^t f(z)>C_1\alpha-C_2$, then $z \in \mathfrak D_{\widetilde{J}+1}$, for $\widetilde{J}>\frac{\log_2 \alpha}{nK_\calS(t-1)}$, which implies
\begin{align*}
|E|&
=\left| \left\{x \in [0, 1]^n: \left|\mathbb A_{\calS}^t f(x) \right|>C_1\alpha-C_2 \right\} \right| \\
& \lesssim (1-\eta)^{\frac{\log_2 \alpha}{nK_\calS(t-1)}} = \alpha^{\frac{\log_2(1-\eta)}{nK_{\calS}(t-1)}}. 
\end{align*}
Thus, this gives
$$
(C_1\alpha-C_2) |E|^{\frac{nK_{\calS}(t-1)}{-\log_2(1-\eta)}} \lesssim \alpha \cdot \left(\alpha^{\frac{\log_2(1-\eta)}{nK_{\calS}(t-1)}} \right)^{\frac{nK_{\calS}(t-1)}{-\log_2(1-\eta)}} \simeq 1, 
$$
which concludes the desired weak-type bound \eqref{20251227eq80}. 

\medskip 

\noindent {\bf Step II: Strong-type bounds within the off-critical line regime \eqref{20251228eq01}: off-critical line estimates. } Let $(p, q)$ be a pair satisfying 
\begin{equation} \label{20251228eq11}
\frac{1}{q}-\frac{1}{p}>\frac{nK_{\calS}(t-1)}{-\log_2(1-\eta)},
\end{equation} 
which, in particular, gives $p>q$. Without loss of generality, we may assume that $p<\infty$, as the case $p=\infty$ is analogous and we would like to leave the details to the interested reader.

Our goal in the second step is to show that $\mathbb A_{\calS}^t f$ extends a bounded operator from $L^p(\R^n)$ to $L^q(\R^n)$. Decompose
$$
\mathbb A_{\calS}^t =\sum_{j \ge 0} \mathbb A_{\calS^{(j)}}^t, 
$$
where 
$$
\mathbb A_{\calS^{(j)}}^t f(x):=\sum_{Q \in \calS^{(j)}} \frac{\one_Q(x)}{|Q|^{t}}\int_Q |f(y)|\,dy
$$
It suffices to show that there exists some absolute constant $C_3>0$, such that for each $j \ge 1$, 
\begin{equation}  \label{20251228eq12}
\left\|\mathbb A_{\calS^{(j)}}^t \right\|_{L^p(\R^n) \to L^q(\R^n)} \lesssim 2^{-C_3j}. 
\end{equation} 
Indeed, since the dyadic cubes in $\calS^{(j)}$ are mutually disjoint, we have 
\begin{align} \label{20251228eq02}
    \left\|\mathbb A_{\calS^{(j)}}^tf \right\|^{q}_{L^{q}(\R^n)}
&= \int_{[0, 1]^n}  \left|\sum_{Q \in \calS^{(j)}}  \frac{\one_Q(x)}{|Q|^{t}}\int_Q |f(y)|\,dy \right|^{q} dx \nonumber  \\
&= \sum_{Q \in \calS^{(j)}} |Q|^{1-tq} \left(\int_{Q} |f(y)| dy \right)^{q} \nonumber \\
&=\sum_{Q \in \calS^{(j)}}|Q|^{1-tq} |Q|^{q}\left(\frac{1}{|Q|}\int_{Q} |f(y)| dy \right)^{q} \nonumber \\
& \le \sum_{Q \in \calS^{(j)}}|Q|^{1-tq} |Q|^{q}\left(\frac{1}{|Q|}\int_{Q} |f(y)|^p dy \right)^{\frac{q}{p}} \nonumber \\
&= \sum_{Q \in \calS^{(j)}}|Q|^{q(1-t)} \cdot |Q|^{\frac{p-q}{p}} \left(\int_{Q} |f(y)|^p dy \right)^{\frac{q}{p}},
\end{align}
where $p'$ is the conjugate of $p$ satisfying $1/p+1/p'=1$.
Since $\calS$ is a graded family with degree $K_{\calS}$, we have for each $Q \in \calS^{(j)}$, $|Q| \ge 2^{-jK_{\calS}n}$. Therefore, using the assumption that $t>1$ and \eqref{20251226eq30}, we derive that 
\begin{align*}
\textrm{RHS of \eqref{20251228eq02}} 
& \le 2^{jK_{\calS}nq(t-1)} \sum_{Q \in \calS^{(j)}} |Q|^{\frac{p-q}{p}} \left(\int_{Q} |f(y)|^p dy \right)^{\frac{q}{p}} \\
& \le 2^{jK_{\calS}nq(t-1)}  \left(\sum_{Q \in \calS^{(j)}} |Q| \right)^{\frac{p-q}{p}} \left(\sum_{Q \in \calS^{(j)}} \int_Q |f(y)|^p dy \right)^{\frac{q}{p}} \\
& \le 2^{jK_{\calS}nq(t-1)} \cdot (1-\eta)^{\frac{(p-q)j}{p}} \left\|f \right\|_{L^p(\R^n)}^q \\
&= 2^{\,j\left(K_{\calS}nq(t-1)+\frac{p-q}{p}\log_2(1-\eta)\right)} \left\|f \right\|_{L^p(\R^n)}^q, 
\end{align*}
which gives
\begin{equation} \label{20251228eq04}
\left\| \mathbb A_{\calS^{(j)}}^t \right\|_{L^p(\R^n) \to L^q(\R^n)} \lesssim 2^{j\left(K_{\calS}n(t-1)+\frac{p-q}{pq}\log_2(1-\eta)\right)}.
\end{equation} 
Note that by \eqref{20251228eq11}, we have $K_{\calS}n(t-1)+\frac{p-q}{pq}\log_2(1-\eta)<0$. Hence \eqref{20251228eq12} holds, which completes \textbf{Step II}.

\medskip 

\noindent{\bf Step~III: Weak-type estimate at 
$\left(1/p, 1/q\right)
=\left(\frac{-\log_2(1-\eta)+nK_{\calS}(1-t)}{-\log_2(1-\eta)},\,1\right)$
 via a dyadic version of the Forelli--Rudin method.}
In this step we implement the Forelli--Rudin method. More precisely,  for any $x\in[0,1]^n$, define the counting function ${\bf N}$ and the weight function ${\bf W}$ at $x$ by
$$
{\bf N}(x):=\sup \; \{\,j\in\N:\ x\in \mathfrak D_j\,\} 
\qquad \textrm{and} \qquad 
{\bf W}(x):=2^{\,nK_{\calS}(t-1){\bf N}(x)}, 
$$
respectively. Note that the set $\{x \in [0, 1]^n: {\bf N}(x)=+\infty\}$ has measure zero. We have the following claims. 

\medskip 

\noindent {\bf Claim 1:}
$$
\mathbb A_{\calS}^t f(x) \lesssim {\bf W}(x) M_{\textrm{HL}}^{\calD} f(x), \qquad a.e. \quad x \in [0, 1]^n, 
$$ 
where $M_{\textrm{HL}}^{\calD}$ refers to the standard dyadic Hardy--Littlewood maximal operator associated to $\calD$. Indeed, using the fact that $\calS$ is graded, we have
\begin{equation} \label{20260127eq02}
\mathbb A_{\calS}^t f(x) 
 =\sum_{x \in Q \in \calS} |Q|^{1-t} \cdot \frac{1}{|Q|} \int_Q |f(y)| dy 
 \le M_{\textrm{HL}}^{\calD} f(x) \cdot \sum_{x \in Q \in \calS} |Q|^{1-t}, \qquad a.e. \quad x \in [0, 1]^n. 
\end{equation} 
By the definition of the counting function $N$, we know that for each $x \in [0, 1]^n$, $x \in \mathfrak D_{{\bf N}(x)} \subseteq \dots \subseteq \mathfrak D_0=[0, 1]^n$ and $x$ belongs to at most one dyadic cube in each $\calS^{(j)}, \; 0 \le j \le {\bf N}(x)$. Therefore, by the fact that $\calS$ is graded, 
\begin{equation} \label{20260127eq03}
\sum_{x \in Q \in \calS} |Q|^{1-t} \le \sum_{j=0}^{{\bf N}(x)} 2^{njK_{\calS}(t-1)} \lesssim 2^{n{\bf N}(x)K_{\calS}(t-1)}={\bf W}(x), \qquad a.e. \quad x \in [0, 1]^n.   
\end{equation}
The desired {\bf Claim 1} therefore follows from \eqref{20260127eq02} and \eqref{20260127eq03}. 

\medskip 

\noindent {\bf Claim 2:} Denote 
$$
p=\frac{-\log_2(1-\eta)}{-\log_2(1-\eta)+nK_{\calS}(1-t)} \qquad \textrm{ with} \qquad  p'=\frac{-\log_2(1-\eta)}{nK_{\calS}(t-1)}.
$$ 
Then $W \in L^{p', \infty}(\R^n)$. 

To see the second claim, we observe that 
$$
\{x \in [0, 1]^n: {\bf W}(x) \ge 2^{nK_{\calS}(t-1)j} \}=\{x \in [0, 1]^n: {\bf N}(x) \ge j\} \subseteq \mathfrak D_j,
$$
and therefore, by the contracting property of $\calS$, 
$$
\left| \{x \in [0, 1]^n: {\bf W}(x) \ge 2^{nK_{\calS}(t-1)j} \} \right| \le |\mathfrak D_j|=\sum_{Q \in \calS^{(j)}} |Q| \lesssim (1-\eta)^j. 
$$
Thus, for any $j \ge 0$,
\begin{align*}
2^{nK_{\mathcal S}(t-1)j}
\Bigl|\bigl\{x\in[0,1]^n:\, {\bf W}(x)\ge 2^{nK_{\mathcal S}(t-1)j}\bigr\}\Bigr|^{\frac{1}{p'}}
&\lesssim 2^{nK_{\mathcal S}(t-1)j}(1-\eta)^{\frac{j}{p'}} \\
&\lesssim 2^{nK_{\mathcal S}(t-1)j}\,(1-\eta)^{\,j\cdot\frac{nK_{\mathcal S}(t-1)}{-\log_2(1-\eta)}} =1.
\end{align*}
In general, for any $\lambda>0$, choose $j\ge 0$ such that $\lambda \in \bigl[2^{nK_{\calS}(t-1)j},\,2^{nK_{\calS}(t-1)(j+1)}\bigr)$. Then $\{{\bf W}\ge \lambda\}\subseteq \{{\bf W}\ge 2^{nK_{\calS}(t-1)j}\}$, so the above estimate extends to all $\lambda>0$, which proves {\bf Claim 2}.

\medskip

To this end, using both {\bf Claims 1} and ${\bf 2}$ and Lorentz--H\"older's inequality, we have 
\begin{align} \label{restrictedWTX}
\left\| \mathbb A_{\calS}^t f\right\|_{L^{1, \infty}(\R^n)} &
\lesssim \left\| W \cdot M_{\textrm{HL}}^{\calD} f \right\|_{L^{1, \infty}(\R^n)} \lesssim \left\|W \right\|_{L^{p', \infty}(\R^n)} \left\|M_{\textrm{HL}}^{\calD} f \right\|_{L^p(\R^n)}  \nonumber \\
& \lesssim \left\|W \right\|_{L^{p', \infty}(\R^n)} \left\| f \right\|_{L^p(\R^n)} \lesssim \|f \|_{L^p(\R^n)}. 
\end{align} 
The proof of {\bf Step 3} is complete. 

\medskip 

\noindent {\bf Step IV: Weak-type bounds on the critical line \eqref{20251228eq30}: critical line estimate.} The last part simply follows from an application of the off-diagonal Marcinkiewicz interpolation between \eqref{20251227eq80} and \eqref{restrictedWTX}.

\vspace{0.1in}

The proof is complete. 
\end{proof}

\subsection{Revisiting the endpoint $\left(1/p,1/q\right)=\left(\frac{-\log_2(1-\eta)+nK_{\calS}(1-t)}{-\log_2(1-\eta)},\,1\right)$ using sparse domination and Bourgain's interpolation trick} \label{20260127subsec01}

In the final part of this section, we revisit \textbf{Step~3} in the proof of Theorem~\ref{20251227thm01}. We present two different approaches: one via sparse domination (built on  \cite[Theorem~E]{CCDO2017}) and the other via Bourgain's interpolation trick. Although both approaches yield only a \emph{restricted weak-type estimate} at this upper-right endpoint, they still suffice to obtain weak-type bounds at other points on the critical line, namely,
\[
\left\{\left(\frac{1}{p},\frac{1}{q}\right)\in[0,1]^2:\ 
\frac{1}{q}-\frac{1}{p}=\frac{nK_{\calS}(t-1)}{-\log_2(1-\eta)},\ \ q\neq 1\right\}.
\]
We expect these methods to have further applications to other problems concerning $\mathbb A_{\calS}^t$, for instance, to weighted estimates along the critical line.

\subsubsection{Restricted weak-type estimate at $\left(1/p,1/q\right)=\left(\frac{-\log_2(1-\eta)+nK_{\calS}(1-t)}{-\log_2(1-\eta)},\,1\right)$ via sparse domination} \label{20260127subsubsection01}
Denote as usual
$$
p=\frac{-\log_2(1-\eta)}{-\log_2(1-\eta)+nK_{\calS}(1-t)} \quad \textrm{with} \quad p'=\frac{-\log_2(1-\eta)}{nK_{\calS}(t-1)}, 
$$
and let $f \in L^{p, 1}(\R^n)$. Without loss of generality, we may assume $\textrm{supp} f \subseteq [0, 1]^n$. By \cite[Exercise 1.4.14]{Grafakos2014} and using the assumption that ${\mathbb A}_{\calS}^t$ is localized in $[0, 1]^n$, we have 
\begin{equation} \label{20260126eq01}
\left\| {\mathbb A}_{\calS}^t f \right\|_{L^{1, \infty}(\R^n)} \simeq \sup_{\substack{E \subseteq [0, 1]^n \\ |E|>0}} \inf_{\substack{E' \subseteq E \\ |E| \le 2|E'|}} \left| \int_{E'} {\mathbb A}_{\calS}^t f(x)dx \right|. 
\end{equation} 
 Then for any measurable $E' \subseteq [0, 1]^n$, one has 
\begin{align} \label{20260126eq02}
 \int_{E'} {\mathbb A}_{\calS}^t f(x)dx 
 &= \sum_{Q \in \calS} |Q| \left(\frac{1}{|Q|^t} \int_Q |f(x)|dx \right) \left( \frac{1}{|Q|} \int_{Q} \one_{E'}(x) dx \right)  \nonumber \\
 & = \sum_{Q \in \calS} |Q| \left(\frac{1}{|Q|} \int_Q |f(x)|dx \right) \left( \frac{1}{|Q|^t} \int_{Q} \one_{E'}(x) dx \right) \nonumber \\
 & \lesssim \sum_{Q \in \calS} |E(Q)| \left(\frac{1}{|Q|} \int_Q |f(x)|dx \right) \left( \frac{1}{|Q|^t} \int_{Q} \one_{E'}(x) dx \right) \nonumber  \\
 & \lesssim \int_{[0, 1]^n} M_{\textrm{HL}}^{\calD} f(x) M_t^{\calS} \one_{E'}(x)dx, 
\end{align} 
where
$$
M_t^{\calS} f(x):=\sup_{Q \in \calS} \frac{\one_{Q}(x)}{|Q|^t} \int_Q |f(x)|dx
$$
is the \emph{hypersingular maximal operator} associated to the graded sparse family $\calS$. Observe that 
\begin{equation} \label{20260127eq50}
M_t^{\calS}: L^\infty(\R^n) \to L^{p', \infty}(\R^n). 
\end{equation} 
Indeed, this follows directly by {\bf Step I} in Theorem~\ref{20251227thm01} and the pointwise sparse bound $M_t^{\calS} f(x) \lesssim \mathbb A_{\calS}^t f(x), \; x \in \R^n$. Therefore, by \eqref{20260126eq02},  Lorentz--H\"olider's inequality, and the $L^{p, 1}(\R^n)$ boundedness\footnote{This is a standard consequence of the off-diagonal Marcinkiewicz interpolation theorem applied to  $M_{\mathrm{HL}}^{\mathcal D}$ (see, e.g.,  \cite[Theorem~1.4.19]{Grafakos2014}).}
 of $M_{\textrm{HL}}^{\calD}$, we have
\begin{align*}
\int_{E'} \mathbb A_{\calS}^t f(x)dx &\lesssim \left\|M_{\textrm{HL}}^{\calD} f \right\|_{L^{p, 1}(\R^n)} \left\|M_t^{\calS} \one_{E'} \right\|_{L^{p', \infty}} \\
& \lesssim \|f\|_{L^{p, 1}(\R^n)} \left\|\one_{E'} \right\|_{L^\infty(\R^n)} \lesssim \|f\|_{L^{p, 1}(\R^n)}. 
\end{align*}
Combining the above estimate with  \eqref{20260126eq01}, we derive that
\begin{equation} \label{20260127eq67}
\left\|\mathbb A_{\calS}^t \right\|_{L^{1, \infty}(\R^n)} \lesssim \left\|f \right\|_{L^{p, 1}(\R^n)}, 
\end{equation} 
which gives the desired result. 
\medskip

\subsubsection{Restricted weak-type estimate at $\left(1/p,1/q\right)=\left(\frac{-\log_2(1-\eta)+nK_{\calS}(1-t)}{-\log_2(1-\eta)},\,1\right)$ via Bourgain's interpolation trick} 

We begin by recalling Bourgain's interpolation lemma (see, e.g., \cite{Bourgain1985,BS2011}; see also \cite{CSWW1999} for an abstract extension in the setting of fairly general normed vector spaces).

\begin{lem}[Bourgain's interpolation trick]\label{Bourgain} 
Let $\beta_1, \beta_2>0$ and $\{T_j\}_{j \ge 0}$ be a collection of sublinear operators satisfying
$$
\left\|T_j \right\|_{L^{p_1}(\R^n) \to L^{q_1}(\R^n)} \le M_1 2^{\beta_1 j}
$$
and
$$
\left\|T_j \right\|_{L^{p_2}(\R^n) \to L^{q_2}(\R^n)} \le M_2 2^{-\beta_2 j}, 
$$
for some $1 \le p_1, p_2, q_1, q_2 \le +\infty$ and $M_1, M_2>0$, then $T=\sum_{j \ge 1} T_j$ enjoys restricted weak type estimate between the intermediate spaces:
$$
\left\|T \right\|_{L^{p, 1}(\R^n) \to L^{q, \infty}(\R^n)} \le C M_1^{\theta} M_2^{1-\theta}, 
$$
where
$$
\theta=\frac{\beta_2}{\beta_1+\beta_2}, \quad \frac{1}{p}=\frac{\theta}{p_1}+\frac{1-\theta}{p_2}, \quad \textrm{and} \quad \frac{1}{q}=\frac{\theta}{q_1}+\frac{1-\theta}{q_2}, 
$$
and $C$ depends only on $\beta_1$ and $\beta_2$. 
\end{lem}

\medskip 

To apply Lemma~\ref{Bourgain}, we first note that, on the one hand,
letting $p=\infty$ and $q=1$ in \eqref{20251228eq04}, we see that 
$$
\left\| \mathbb A_{\calS^{(j)}}^t \right\|_{L^\infty(\R^n) \to L^1(\R^n)} \lesssim 2^{j\left(K_{\calS}n(t-1)+\log_2(1-\eta)\right)}. 
$$
Observe that $nK_{\calS}(t-1)+\log_2(1-\eta)<0$, which follows from the assumption $t<1-\frac{\log_2(1-\eta)}{nK_{\calS}}$. On the other hand,  we have 
\begin{align*}
    \left\|\mathbb A_{\calS^{(j)}}^tf \right\|_{L^1(\R^n)}
&= \int_{\R^n}  \left|\sum_{Q \in \calS^{(j)}}  \frac{\one_Q(x)}{|Q|^{t}}\int_Q |f(y)|\,dy \right|dx \\
&= \sum_{Q \in \calS^{(j)}} |Q|^{1-t} \int_{Q} |f(y)| dy  \\
& \le 2^{jK_{\calS}n(t-1)} \sum_{Q \in \calS^{(j)}}  \int_{Q} |f(y)| dy \\
& \le 2^{jK_{\calS}n(t-1)} \left\|f \right\|_{L^1(\R^n)}, 
\end{align*}
which gives
\begin{equation} \label{20251226eq11}
\left\|\mathbb A_{\calS^{(j)}}^t  \right\|_{L^{1}(\R^n) \to L^{1}(\R^n)} \lesssim 2^{jK_{\calS}n(t-1)}. 
\end{equation}
Applying now Lemma \ref{Bourgain} with
$$
\beta_1=K_Sn(t-1), \quad \beta_2=-nK_{\calS}(t-1)-\log_2(1-\eta), \quad p_1=q_1=q_2=1, \quad \textrm{and} \quad p_2=\infty,
$$
which gives the boundedness of 
\begin{equation} \label{restrictedWT}
\mathbb A_{\calS}^t: L^{p, 1}(\R^n) \to L^{q, \infty}(\R^n), 
\end{equation} 
with
$$
\theta=\frac{\beta_2}{\beta_1+\beta_2}=\frac{-\log_2(1-\eta)+nK_{\calS}(1-t)}{-\log_2(1-\eta)}, \qquad \frac{1}{q}=1, 
$$
and
$$
\frac{1}{p}=\frac{\theta}{p_1}+\frac{1-\theta}{p_2}=\frac{-\log_2(1-\eta)+nK_{\calS}(1-t)}{-\log_2(1-\eta)}.
$$
The proof is complete. 

\medskip 

\section{Some open problems} \label{20251219sec06}

\noindent\textbf{1. Upgrading restricted weak-type to weak-type estimates at the right-upper endpoint via sparse domination.} It would be interesting to return to the argument in Section~\ref{20260127subsubsection01} and explore whether the stopping time argument in \cite[Theorem~E]{CCDO2017} can be adapted to $\mathbb A_{\calS}^t$, thereby upgrading \eqref{20260127eq67} to a weak-type estimate. Such an improvement would have further applications to the weighted theory for $\mathbb A_{\calS}^t$.

\medskip

\noindent\textbf{2. Further applications of Forelli-Rudin method on critical lines.} It would be desirable to further explore the Forelli-Rudin methods developed in this work in establishing critical-line estimates for other hypersingular operators. For instance:

\smallskip

\noindent\textbf{(a) Forelli--Rudin type operators.}
For $a,b,c\in\R$, define
\[
T_{a,b,c}f(z)=(1-|z|^2)^a\int_{\B_n}\frac{(1-|w|^2)^b}{(1-\langle z,w\rangle)^c}\,f(w)\,dV(w),
\]
and
\[
S_{a,b,c}f(z)=(1-|z|^2)^a\int_{\B_n}\frac{(1-|w|^2)^b}{|1-\langle z,w\rangle|^c}\,f(w)\,dV(w),
\]
where $\B_n$ is the unit ball in $\C^n$ and $dV$ is the normalized volume measure on $\B_n$.
In \cite{ZZ2022}, Zhao and Zhou characterized the strong-type $L^p_\alpha(\B_n)\to L^q_\beta(\B_n)$ bounds for these operators under various assumptions on the parameters $a,b,c,\alpha$, and $\beta$.
Here, for $1\le p<\infty$ and $-1<\alpha<\infty$, the space $L^p_\alpha(\B_n)$ is defined with respect to the measure $dv_\alpha(z)=c_\alpha(1-|z|^2)^\alpha\,dV(z)$, where $c_\alpha$ is chosen so that $v_\alpha(\B_n)=1$.
A natural question is whether the method developed in the present paper can be applied to obtain critical-line estimates for Forelli--Rudin type operators in the hypersingular regime (see, e.g., \cite[Theorem~1.1]{ZZ2022}).

\medskip

\noindent\textbf{(b) Other hypersingular averaging operators and forms.}
One may consider an $r$-th mean variant of the hypersingular sparse operator: for $t>1$ and $r\ge1$, define
\[
\mathbb A_{\calS,r}^t f(x)
:=\sum_{Q\in\calS}\frac{\one_Q(x)}{|Q|^{t-1}}\langle |f|\rangle_{Q,r}
=\sum_{Q\in\calS}\frac{\one_Q(x)}{|Q|^{t-1+\frac1r}}\Bigl(\int_Q |f|^r\Bigr)^{1/r},
\]
where $\langle |f|\rangle_{Q,r}:=\bigl(|Q|^{-1}\int_Q |f|^r\bigr)^{1/r}$ and $\calS$ is a graded sparse family in $\R^n$.
More generally, for $t>1$ and $1\le r,s<\infty$, one may introduce the \emph{$(r,s)$-hypersingular sparse form}
\[
\Lambda^t_{\calS;r,s}(f_1,f_2)
:=\sum_{Q\in\calS}|Q|^{2-t}\,\langle |f_1|\rangle_{Q,r}\,\langle |f_2|\rangle_{Q,s}.
\]
It would be interesting to understand how the parameters $t,r,s$, together with the geometry of $\calS$, affect the boundedness properties of these operators and forms.

\end{document}